\documentclass{article}
\usepackage{amsmath,amsthm,mathtools,mathrsfs,amsfonts,amssymb,stmaryrd,xy}
 \usepackage{mathabx}  
\usepackage[a4paper]{geometry}
\theoremstyle{definition}
\newtheorem{remark}{Remark}[section]
\newtheorem{remarks}[remark]{Remarks}

 \theoremstyle{plain}
\newtheorem{definition}[remark]{Definition}
\newtheorem{theorem}[remark]{Theorem}
\newtheorem{proposition}[remark]{Proposition}
\newtheorem{corollary}[remark]{Corollary}
\newtheorem{lemma}[remark]{Lemma}

\newtagform{empty}{}{}

\DeclareMathOperator{\Id}{id}

\DeclareMathOperator{\Ad}{Ad}

\DeclareMathOperator{\Ind}{Ind}

\DeclareMathOperator{\Mor}{Mor}

\newcommand{\smalldiagram}{}

\newcommand{\hDelta}{\widehat{\Delta}}

\newcommand{\halpha}{\widehat{\alpha}}
\newcommand{\hdelta}{\widehat{\delta}}

\newcommand{\lnspan}{\big[} \newcommand{\rnspan}{\big]}

\newcommand{\frakB}{\mathfrak{B}}
\newcommand{\frakBo}{\frakB^{\dag}}
\newcommand{\frakC}{\mathfrak{C}}

\newcommand{\frakH}{\mathfrak{H}}
\newcommand{\frakK}{\mathfrak{K}}

\newcommand{\cbasel}[2]{(\mathfrak{#1}, \mathfrak{#2}, \mathfrak{#1}^{\dag})}
\newcommand{\cbases}[2]{{_{\mathfrak{#1}}}\mathfrak{#2}_{\mathfrak{#1}^{\dag\!}}} 
\newcommand{\cbaseos}[2]{{_{\mathfrak{#1}^{\dag\!}}}\mathfrak{#2}_{\mathfrak{#1}}} 
\newcommand{\cbasesb}{\cbases{B}{H}}
\newcommand{\cbasesc}{\cbases{C}{K}}
\newcommand{\cbaseosc}{\cbaseos{C}{K}}
\newcommand{\cbaseosb}{\cbaseos{B}{H}}

\newcommand{\lt}{\smalltriangleleft}
\newcommand{\rt}{\smalltriangleright}

\newcommand{\hbeta}{\widehat{\beta}}
\newcommand{\hA}{\widehat{A}} \newcommand{\ha}{\widehat{a}}

\newcommand{\mycong}{\xrightarrow{\cong}}

\newcommand{\rtensor}[3]{ {_{#1}\! \underset{#2}{\otimes}\! {}_{#3}}}

\newcommand{\htensor}[2]{\rtensor{#1}{\frakH}{#2}}

\newcommand{\rtensorab}{\htensor{\alpha}{\beta}}
\newcommand{\rtensorcb}{\htensor{\gamma}{\beta}}
\newcommand{\rtensorca}{\htensor{\gamma}{\alpha}}

\newcommand{\rtensorh}{\underset{\frakH}{\otimes}}

\newcommand{\fibre}[3]{ {_{#1}\! \underset{#2}{\ast}\! {}_{#3}}}

\newcommand{\rfibre}[2]{ {_{#1}\ast_{#2}}}
\newcommand{\hfibre}[2]{ {_{#1}\! \underset{\frakH}{\ast}\! {}_{#2}}}

\newcommand{\hfibrecb}{ {_{\gamma}\! \underset{\frakH}{\ast}\! {}_{\beta}}}

\newcommand{\fsource}{\htensor{\hbeta}{\alpha}}
\newcommand{\frange}{\rtensorab}

\newcommand{\sfsource}{\fsource}
\newcommand{\sfrange}{\frange}

\newcommand{\tl}{\ensuremath \olessthan}
\newcommand{\tr}{\ensuremath \ogreaterthan}

\newcommand{\Hsource}{H \fsource H}
\newcommand{\Hrange}{H \frange H}

\newcommand{\checkHsource}{H \htensor{\halpha}{\hbeta} H}
\newcommand{\checkHrange}{H \htensor{\hbeta}{\alpha} H}
\newcommand{\hatHsource}{H \htensor{\alpha}{\beta} H}
\newcommand{\hatHrange}{H \htensor{\beta}{\halpha} H}

\newcommand{\checkV}{\widecheck{V}}
\newcommand{\hatV}{\widehat{V}}
\newcommand{\sHsource}{H \sfsource H}
\newcommand{\sHrange}{H \sfrange H}

\newcommand{\HfibreK}{H \rtensorab K}

\newcommand{\Hone}{H \sfsource H \sfsource H}

\newcommand{\Htwo}{H \sfrange H \sfsource  H}

\newcommand{\Hthree}{H \sfrange H \sfrange H}

\newcommand{\Hfour}{(\sHsource) \rtensor{\alpha \lt \alpha}{\frakH}{\beta} H}

\newcommand{\Hfive}{H \rtensor{\hbeta}{\frakH}{\alpha \rt \alpha} (\sHrange)}

\newcommand{\Hfourlt}{H \sfsource H \rtensor{\beta}{\frakH}{\alpha} H}
\newcommand{\Hfourrt}{\big(\sHrange\big) \rtensor{\hbeta \lt
    \beta}{\frakH}{\alpha} H}

\newcommand{\fibreab}{\fibre{\alpha}{\frakH}{\beta}}
\newcommand{\AfibreA}{A \fibreab A}

\newcommand{\AfibreB}{A \fibreab B}

\newcommand{\kalpha}[1]{|\alpha\rangle_{\leg{#1}}}
\newcommand{\balpha}[1]{\langle\alpha|_{\leg{#1}}}
\newcommand{\kbeta}[1]{|\beta{}\rangle_{\leg{#1}}}
\newcommand{\bbeta}[1]{\langle\beta|_{\leg{#1}}}
\newcommand{\khbeta}[1]{|\hbeta{}\rangle_{\leg{#1}}}
\newcommand{\bhbeta}[1]{\langle\hbeta|_{\leg{#1}}}
\newcommand{\khalpha}[1]{|\halpha{}\rangle_{\leg{#1}}}
\newcommand{\bhalpha}[1]{\langle\halpha|_{\leg{#1}}}

\newcommand{\cfact}{\ensuremath \mathrm{C}^{*}\mathrm{\text{-}fact}}

\newcommand{\leg}[1]{[#1]}

\newcommand{\hAbimod}{({_{\alpha}\hA_{\hbeta}},\hDelta)}
\newcommand{\Abimod}{({_{\beta}A_{\alpha}},\Delta)}

\title{$C^{*}$-pseudo-Kac systems and duality for coactions of
  concrete Hopf $C^{*}$-bimodules}

\author{Thomas Timmermann\\[1ex]
  \texttt{timmermt@math.uni-muenster.de}\\ SFB 478
  ``Geometrische Strukturen in der Mathematik''\\ Hittorfstr.\
  27, 48149 M\"unster}

\date{\today}

\begin{document}
 \xyrequire{matrix} \xyrequire{arrow}

\maketitle 

\abstract{
  We study coactions of concrete Hopf $C^{*}$-bimodules in the
  framework of (weak) $C^{*}$-pseudo-Kac systems, define reduced
  crossed products and dual coactions, and prove an analogue of
  Baaj-Skandalis duality. 
}

\section{Introduction}

In a seminal article \cite{baaj:2}, Baaj and Skandalis developed a
duality theory for coactions of Hopf $C^{*}$-algebras on
$C^{*}$-algebras that extends the Takesaki-Takai duality of
actions of locally compact abelian groups to all locally compact
groups and, more generally, to all regular locally compact quantum
groups. More precisely, Baaj and Skandalis introduce the notion of a
Kac system which  consists of a regular multiplicative unitary and an
additional symmetry,  consider coactions of the two Hopf  $C^{*}$-algebras
(the ``legs'') of the multiplicative unitary, define for every coaction
of each leg a reduced crossed product that carries a coaction of
the other leg, and show that two applications of this construction
yield a stabilization of the original coaction.

In this article, we extend the duality theory of Baaj and Skandalis to
coactions of concrete Hopf $C^{*}$-bimodules, applying the methods and
concepts introduced in \cite{timmer:cpmu} to the
constructions in \cite{baaj:2}. In particular, our theory
covers coactions of the Hopf $C^{*}$-bimodules associated to a locally compact
groupoid.   An article on examples is in preparation.

Let us mention that a similar duality theory like the one presented
here was developed in the PhD thesis of the author
\cite{timmermann}. Our new approach allows us to drop a
rather restrictive condition (decomposability) needed in
\cite{timmermann}, and to work in the framework of $C^{*}$-algebras
instead of the somewhat exotic $C^{*}$-families. Moreover, the
approach presented here greatly simplifies the complex assumptions
needed \cite{timmermann}.

This work was supported by the SFB 478 ``Geometrische Strukturen in
der Mathematik'' which is funded by the Deutsche
  Forschungsgemeinschaft (DFG).

\paragraph{Organization} This article is organized as follows:

First, we fix notation and terminology, and summarize some background
on (Hilbert) $C^{*}$-modules.

In Section 2, we recall the definition of
$C^{*}$-pseudo-multiplicative unitaries and concrete Hopf
$C^{*}$-bimodules given in \cite{timmer:cpmu}.

In Section 3, we introduce weak $C^{*}$-pseudo-Kac systems which
provide the framework for the construction of reduced crossed
products. 

In Section 4, we consider coactions of concrete Hopf
$C^{*}$-bimodules and construct reduced crossed products and duals for
such coactions.

In Section 5, we introduce $C^{*}$-pseudo-Kac systems and establish an
analogue of Baaj-Skandalis duality for coactions of concrete Hopf
$C^{*}$-bimodules.

In Section 6, we associate to each locally compact groupoid a
$C^{*}$-pseudo-Kac system.

\paragraph{Preliminaries}

Given a subset $Y$ of a normed space $X$, we denote by $[Y] \subseteq X$
the closed linear span of $Y$.

Given a Hilbert space $H$ and a subset $X \subseteq {\cal L}(H)$, we
denote by $X'$  the commutant of $X$. Given  Hilbert spaces $H$, $K$,
a $C^{*}$-subalgebra $A \subseteq {\cal L}(H)$, and a $*$-homomorphism
$\pi \colon A \to {\cal L}(K)$, we put
\begin{align*}
  {\cal L}^{\pi}(H,K) := \{ T \in {\cal L}(H,K) \mid Ta = \pi(a)T
  \text{ for all } a \in A\};
\end{align*}
thus, for example, $A' = {\cal L}^{\Id_{A}}(H)$.

We shall make extensive use of (right) $C^{*}$-modules, also known as
Hilbert $C^{*}$-modules or Hilbert modules. A standard reference is
\cite{lance}.

All sesquilinear maps like inner products of Hilbert spaces
or $C^{*}$-modules are assumed to be conjugate-linear in the first
component and linear in the second one.

Let $A$ and $B$ be $C^{*}$-algebras.  
Given $C^{*}$-modules $E$ and $F$ over $B$, we denote the space of all
adjointable operators $E\to F$ by ${\cal L}_{B}(E,F)$.

Let $E$ and $F$ be $C^{*}$-modules over $A$ and $B$, respectively, and
let $\pi \colon A \to {\cal L}_{B}(F)$ be a $*$-homomorphism. Then one
can form the internal tensor product $E \otimes_{\pi} F$, which is a
$C^{*}$-module over $B$ \cite[Chapter 4]{lance}. This $C^{*}$-module
is the closed linear span of elements $\eta \otimes_{A} \xi$, where
$\eta \in E$ and $\xi \in F$ are arbitrary, and $\langle \eta
\otimes_{\pi} \xi|\eta' \otimes_{\pi} \xi'\rangle = \langle
\xi|\pi(\langle\eta|\eta'\rangle)\xi'\rangle$ and $(\eta \otimes_{\pi}
\xi)b=\eta \otimes_{\pi} \xi b$ for all $\eta,\eta' \in E$, $\xi,\xi'
\in F$, and $b \in B$.  We denote the internal tensor product by
``$\tr$''; thus, for example, $E \tr_{\pi} F=E \otimes_{\pi} F$. If
the representation $\pi$ or both $\pi$ and $A$ are understood, we
write ``$\tr_{A}$'' or  ``$\tr$'', respectively, instead of
$"\tr_{\pi}$''. 

Given $E$, $F$ and $\pi$ as above, we define a {\em flipped internal
  tensor product} $F {_{\pi}\tl} E$ as follows. We equip the algebraic
tensor product $F \odot E$ with the structure maps $\langle \xi \odot
\eta | \xi' \odot \eta'\rangle := \langle \xi| \pi(\langle
\eta|\eta'\rangle) \xi'\rangle$, $(\xi \odot \eta) b := \xi b \odot
\eta$, and by factoring out the null-space of the semi-norm
$\zeta\mapsto \| \langle \zeta|\zeta\rangle\|^{1/2}$ and taking
completion, we obtain a $C^{*}$-$B$-module $F {_{\pi}\tl} E$.  This is
the closed linear span of elements $\xi {_{\pi}\tl} \eta$, where $\eta \in E$
and $\xi \in F$ are arbitrary, and $\langle \xi {_{\pi}\tl} \eta|\xi'
{_{\pi}\tl} \eta'\rangle = \langle \xi|\pi(\langle\eta|\eta'\rangle)\xi'\rangle$ and
$(\xi {_{\pi}\tl} \eta)b=\xi b {_{\pi}\tl} \eta$ for all $\eta,\eta' \in E$, $\xi,\xi'
\in F$, and $b\in B$. As above, we write ``${_{A}\tl}$'' or simply
``$\tl$'' instead of ``${_{\pi}\tl}$'' if the representation $\pi$ or
both  $\pi$ and $A$ are understood, respectively.

Evidently, the usual and the flipped internal tensor product are related by
a unitary map $\Sigma \colon F \tr E \mycong E \tl F$, $\eta \tr \xi
\mapsto \xi \tl \eta$.

We shall frequently consider the following  kind of $C^{*}$-modules.
Let $H$ and $K$ be Hilbert spaces. We call a subset $\Gamma \subseteq
{\cal L}(H,K)$ a {\em concrete $C^{*}$-module} if $[ \Gamma
\Gamma^{*} \Gamma] = \Gamma$.  If $\Gamma$ is a concrete
$C^{*}$-module, then evidently $\Gamma^{*}$ is a concrete
$C^{*}$-module as well, the space $B:=[ \Gamma^{*}
\Gamma]$ is a $C^{*}$-algebra, and $\Gamma$ is a full right
$C^{*}$-module over $B$ with respect to the inner product given by
$\langle \zeta|\zeta'\rangle=\zeta^{*}\zeta'$ for all $\zeta,\zeta'
\in \Gamma$.

\section{$C^{*}$-pseudo-multiplicative unitaries and concrete Hopf
  $C^{*}$-bimodules}

\label{section:basics}
We recall several constructions and definitions from
\cite{timmer:cpmu} which are fundamental to the duality theory
developed in the following sections.

\paragraph{$C^{*}$-bases and  $C^{*}$-factorizations}
$C^{*}$-bases and $C^{*}$-factorizations are simple but convenient
concepts used in the definition of the $C^{*}$-relative tensor
product; for details, see \cite[Section 2.1]{timmer:cpmu}.  

A {\em $C^{*}$-base} is a triple $\cbasel{B}{H}$, shortly written
$\cbases{B}{H}$, consisting of a Hilbert space $\frakH$ and two
commuting nondegenerate $C^{*}$-algebras $\frakB,\frakB^{\dag}
\subseteq {\cal L}(\frakH)$. We shall mainly be interested in the
following example of $C^{*}$-bases: If $\mu$ is a proper KMS-weight
on a $C^{*}$-algebra $B$, then the triple
$(\pi_{\mu}(B),H_{\mu},J_{\mu}\pi_{\mu}(B)J_{\mu})$ is a $C^{*}$-base,
where $H_{\mu}$ denotes the GNS-space, $\pi_{\mu} \colon B \to {\cal
  L}(H_{\mu})$ the GNS-representation, and $J_{\mu} \colon H_{\mu} \to
H_{\mu}$ the modular conjugation associated to $\mu$.

A {\em $C^{*}$-factorization} of a Hilbert space $H$ with respect to a
$C^{*}$-base $\cbases{B}{H}$ is a closed subspace $\alpha \subseteq
{\cal L}(\frakH,H)$ satisfying $[\alpha^{*}\alpha] =
\frakB$, $[\alpha \frakB] = \alpha$, and $[\alpha
\frakH] = H$.  We denote the set of all $C^{*}$-factorizations
of a Hilbert space $H$ with respect to a $C^{*}$-base $\cbases{B}{H}$
by $\cfact(H;\cbases{B}{H})$.

Let $\alpha$ be a $C^{*}$-factorization of a Hilbert space $H$ with
respect to a $C^{*}$-base $\cbases{B}{H}$. Then $\alpha$
is a concrete $C^{*}$-module and a full right $C^{*}$-module over $\frakB$
with respect to the inner product $\langle \xi|\xi'\rangle
:=\xi^{*}\xi'$.  We shall frequently identify $\alpha \tr \frakH$ with
$H$ via the unitary
\begin{align} \label{eq:rtp-iso}
  \alpha \tr \frakH \mycong H, \quad \xi
      \tr \zeta \mapsto \xi\zeta.
\end{align}
There exists a nondegenerate and faithful representation
$\rho_{\alpha} \colon \frakB^{\dag} \to {\cal L}(\alpha \tr \frakH)
\cong {\cal L}(H)$ such that for all $b^{\dag} \in
\frakB^{\dag}$, $\xi\in \alpha$, $\zeta \in \frakH$,
\begin{align*}
  \rho_{\alpha}(b^{\dag})(\xi \tr \zeta) = \xi \tr
  b^{\dag}\zeta \quad \text{or, equivalently,}  \quad
  \rho_{\alpha}(b^{\dag})\xi\zeta = \xi b^{\dag}\zeta.
\end{align*}

 Let $K$ be a  Hilbert space. Then each unitary  $U\colon H
 \to K$ induces a map
 \begin{align*}
   U_{*} \colon \cfact(H;\cbases{B}{H}) \to \cfact(K;\cbases{B}{H}),  \quad \alpha
   \mapsto U\alpha,
 \end{align*} 
and $\rho_{(U\alpha)}(b^{\dag})=U\rho_{\alpha}(b^{\dag})U^{*}$ for
all $\alpha \in \cfact(H;\cbasesb)$ and $b^{\dag} \in \frakBo$  because
\begin{align} \label{eq:cfact-conjugate}
  \rho_{(U\alpha)}(b^{\dagger}) U\xi \zeta =  U\xi b^{\dagger} \zeta =
  U \rho_{\alpha}(b^{\dagger})\xi\zeta \quad \text{for all } \xi \in
  \alpha, \, \zeta \in \frakH.
\end{align}
Let $\beta$ be a $C^{*}$-factorization of $K$ with respect to
$\cbases{B}{H}$. We put
\begin{align*}
  {\cal L}(H_{\alpha},K_{\beta})&:= \big\{ T \in {\cal L}(H,K)
  \,\big|\, T\alpha \subseteq \beta, \, T^{*}\beta \subseteq \alpha\big\}.
\end{align*}
Evidently, ${\cal L}(H_{\alpha},K_{\beta})^{*} = {\cal
  L}(K_{\beta},H_{\alpha})$. Let $T \in {\cal
  L}(H_{\alpha},K_{\beta})$. Then the map $T_{\alpha} \colon \alpha
\to \beta$ given by $\xi \mapsto T\xi$ is an adjointable operator of
$C^{*}$-modules, $(T_{\alpha})^{*}=(T^{*})_{\beta}$, and $T
\rho_{\alpha}(b^{\dag}) =\rho_{\beta}(b^{\dag})T $ for all
$b^{\dag} \in \frakBo$.

Let $\alpha$ be a $C^{*}$-factorization of a Hilbert space $H$ with
respect to a $C^{*}$-base $\cbases{B}{H}$ and let $\cbasesc$ be
another $C^{*}$-base. We call a $C^{*}$-factorization $\beta \in
\cfact(H;\cbasesc)$ {\em compatible} with $\alpha$, written $\alpha
\perp \beta$, if $\lnspan \rho_{\alpha}(\frakBo)\beta\rnspan =\beta $
and $\lnspan \rho_{\beta}(\frakC^{\dag})\alpha\rnspan=\alpha$. In that
case,   $\rho_{\alpha}(\frakBo)$ and $\rho_{\beta}(\frakC^{\dag})$ commute.
 We put $\cfact(H_{\alpha};\cbasesc):=\{ \beta \in
  \cfact(H;\cbasesc) \mid \alpha \perp \beta\}$.

\paragraph{The $C^{*}$-relative tensor product}
The $C^{*}$-relative tensor product of Hilbert spaces is a symmetrized
version of the internal tensor product of $C^{*}$-modules and a
$C^{*}$-algebraic analogue of the relative tensor product of Hilbert
spaces over a von Neumann algebra. We briefly summarize the definition
and the main properties; for details, see \cite[Section
2.2]{timmer:cpmu}.

Let $H$ and $K$ be Hilbert spaces, $\cbasesb$ a $C^{*}$-base, and
$\alpha \in \cfact(H;\cbasesb)$ and $\beta \in \cfact(K;\cbaseosb)$.
The {\em $C^{*}$-relative tensor product} of $H$ and $K$ with respect
to $\alpha$ and $\beta$ is the internal tensor product
\begin{align*}
   \HfibreK := \alpha \tr \frakH \tl \beta.
\end{align*}
We  frequently identify this Hilbert space with $\alpha
\tr_{\rho_{\beta}} K$ and $H {_{\rho_{\alpha}}\tl} \beta$ via the
isomorphisms
\begin{gather} \label{eq:rtp-space}
      \alpha \tr_{\rho_{\beta}} K \cong \HfibreK \cong
    H {_{\rho_{\alpha}} \tl} \beta, \quad
    \xi \tr \eta \zeta \equiv \xi \tr \zeta \tl \eta \equiv \xi \zeta
    \tl \eta.
\end{gather}
Using these isomorphisms, we define for each $\xi \in \alpha$ and
$\eta \in \beta$ two pairs of adjoint operators
\begin{align*}
  |\xi\rangle_{\leg{1}} \colon K &\to \HfibreK, \ \zeta \mapsto \xi
  \tr \zeta, & \langle \xi|_{\leg{1}}:=|\xi\rangle_{\leg{1}}^{*}\colon
  \xi' \tr \zeta &\mapsto
  \rho_{\beta}(\langle\xi|\xi'\rangle)\zeta, \\
  |\eta\rangle_{\leg{2}} \colon H &\to \HfibreK, \ \zeta \mapsto \zeta
  \tl \eta, & \langle\eta|_{\leg{2}} := |\eta\rangle_{\leg{2}}^{*}
  \colon \zeta \tl\eta &\mapsto \rho_{\alpha}(\langle
  \eta|\eta'\rangle)\zeta.
\end{align*}
We put $\kalpha{1} := \big\{ |\xi\rangle_{\leg{1}} \,\big|\, \xi \in
\alpha\big\}$ and similarly define $\balpha{1}$, $\kbeta{2}$,
$\bbeta{2}$.

Let $L$, $M$ be Hilbert spaces, $\gamma \in \cfact(L;\cbasesb)$,
$\delta \in \cfact(M;\cbaseosb)$, and $S \in {\cal L}(H,L)$, $T \in
{\cal L}(K,M)$.  Using
the isomorphisms \eqref{eq:rtp-space}, we define an operator $S
\rtensorh T\in {\cal L}\big(\HfibreK, \, L
\rtensor{\gamma}{\frakH}{\delta} M\big)$ in the following cases:
\begin{enumerate}
\item if $S \in {\cal L}(H_{\alpha},L_{\gamma})$ and
  $T\rho_{\beta}(b)=\rho_{\delta}(b)T$ for all $b \in \frakB$, then
  $S \rtensorh T := S_{\alpha} \tr T$;
\item if $T \in {\cal L}(K_{\beta},M_{\delta})$ and
  $S\rho_{\alpha}(b^{\dag}) = \rho_{\gamma}(b^{\dag})S$ for all
  $b^{\dag} \in \frakBo$, then $S \rtensorh T := S \tl T_{\beta}$.
\end{enumerate}
If $S \in {\cal L}(H_{\alpha},L_{\gamma})$ and $T \in {\cal
  L}(K_{\beta},M_{\delta})$, then both times $S \rtensorh T = S_{\alpha} \tr
\Id_{\frakH} \tl T_{\beta}$.  Put
\begin{align*}
  S_{\leg{1}} &:=  S \rtensorh \Id \colon\HfibreK \to L
  \rtensor{\gamma}{\frakH}{\beta} K, &
  T_{\leg{2}} &:= \Id \rtensorh T \colon  \HfibreK \to H
  \rtensor{\alpha}{\frakH}{\delta} M.
\end{align*}
Given $C^{*}$-algebras $C$, $D$ and $*$-homomorphisms $\rho \colon C
\to \rho_{\alpha}(\frakBo)' \subseteq {\cal L}(H)$, $\sigma
\colon D \to \rho_{\beta}(\frakB)' \subseteq {\cal L}(K)$, we 
define $*$-homomorphisms 
$\rho_{\leg{1}} \colon C \to {\cal L}(\HfibreK)$, $c \mapsto
  \rho(c)_{\leg{1}}$, and 
$\sigma_{\leg{2}} \colon D\to {\cal L}(\HfibreK)$, $d \mapsto
\sigma(d)_{\leg{2}}$.

Let $\cbasesc$ and $\cbases{D}{L}$ be $C^{*}$-bases. Then there exist
compatibility-preserving maps
    \begin{align*}
      \cfact(H_{\alpha};\cbasesc) &\to
      \cfact(\HfibreK;\cbasesc), & \gamma &\mapsto \gamma \lt
      \beta:= \lnspan
      \kbeta{2}\gamma\rnspan, \\
      \cfact(K_{\beta};\cbases{D}{L}) &\to
      \cfact(\HfibreK;\cbases{D}{L}), & \delta &\mapsto \alpha \rt
      \delta:= \lnspan \kalpha{1}\delta\rnspan,
    \end{align*}
 and   for all $\gamma \in
    \cfact(H_{\alpha};\cbasesc)$ and $\delta \in
    \cfact(K_{\beta};\cbases{D}{L})$,
    \begin{align*}
   \rho_{(\gamma \lt \beta)}
    &=(\rho_{\gamma})_{\leg{1}}, & \rho_{(\alpha \rt \delta)} &=
    (\rho_{\delta})_{\leg{2}}, & \gamma \lt \beta &\perp \alpha
    \rt \delta.
    \end{align*}

    The $C^{*}$-relative tensor product is symmetric, functorial and
    associative in a natural sense \cite[Section
    2.2]{timmer:cpmu}. Moreover, if the $C^{*}$-base $\cbasesb$ arises
    from a proper KMS-weight $\mu$ on a $C^{*}$-algebra $B$, then
    $\HfibreK$ can be identified with a von Neumann-algebraic relative
    tensor product, also known as Connes' fusion; see \cite[Section
    2.3]{timmer:cpmu}.

\paragraph{The spatial fiber product of $C^{*}$-algebras and concrete
  Hopf $C^{*}$-bimodules}
The spatial fiber product of $C^{*}$-algebras is a $C^{*}$-algebraic
analogue of the fiber product of von Neumann algebras
\cite{sauvageot:2} and of the relative tensor product of
$C_{0}(X)$-algebras \cite{blanchard}. We briefly summarize the
definition and main properties; for details, see \cite[Section
3]{timmer:cpmu}.  Throughout this paragraph, let $\cbasesb$ be a
$C^{*}$-base.

A {\em (nondegenerate) concrete $C^{*}$-$\cbasesb$-algebra}
$(H,A,\alpha)$ consists of a Hilbert space $H$, a (nondegenerate)
$C^{*}$-algebra $A \subseteq {\cal L}(H)$, and a $C^{*}$-factorization
$\alpha \in \cfact(H;\cbasesb)$ such that $\rho_{\alpha}(\frakBo)A
\subseteq A$. If $(H,A,\alpha)$ is a nondegenerate concrete
$C^{*}$-$\cbasesb$-algebra, then $A' \subseteq
\rho_{\alpha}(\frakBo)'$.

Given a Hilbert space $H$ and a $C^{*}$-algebra $A \subseteq {\cal
  L}(H)$, we put
  \begin{align*}
    \cfact(A;\cbasesb):=\{ \beta \in \cfact(H;\cbasesb) \mid (H,A,\beta) \text{
      is a concrete $C^{*}$-$\cbasesb$-algebra}\}.
  \end{align*}
  Let $\alpha \in \cfact(A;\cbasesb)$ and let $\cbasesc$ be a
  $C^{*}$-base.  We put $\cfact(A_{\alpha};\cbasesc):=\{ \beta \in
  \cfact(A;\cbasesc) \mid \beta \perp \alpha\}$.  If $\beta \in
  \cfact(A_{\alpha};\cbasesc)$ (and $A$ is nondegenerate), we call
  $(H,A,\alpha,\beta)$ a {\em (nondegenerate) concrete
    $C^{*}$-$\cbasesb$-$\cbasesc$-algebra}.

  Let $(H,A,\alpha)$ be a concrete $C^{*}$-$\cbasesb$-algebra and
  $(K,B,\beta)$ a concrete $C^{*}$-$\cbaseosb$-algebra. The {\em fiber
    product} of $(H,A,\alpha)$ and $(K,B,\beta)$ is the
  $C^{*}$-algebra
  \begin{align*}
    \AfibreB := \big\{ T \in {\cal L}(\HfibreK) \,\big|\, &T\kalpha{1},
    T^{*}\kalpha{1} \subseteq [ \kalpha{1} B] \\ &\text{and }
    T\kbeta{2}, T^{*}\kbeta{2} \subseteq [\kbeta{2}A]\big\}.
  \end{align*}
 We do not know whether $\AfibreB$ is nondegenerate if
  $A$ and $B$ are nondegenerate. 
  If $\cbasesc$ is a $C^{*}$-base, then $\gamma \lt \beta \in
  \cfact(\AfibreB;\cbasesc)$ for each $ \gamma \in
  \cfact(A_{\alpha};\cbasesc)$ and $\alpha \rt \delta \in
  \cfact(\AfibreB;\cbasesc)$ for each $\delta \in
  \cfact(B_{\beta};\cbasesc)$.

We do not expect the fiber product to be associative.  
Given $C^{*}$-bases $\cbasesb$ and $\cbasesc$, a concrete
 $C^{*}$-$\cbasesb$-algebra $(H,A,\alpha)$,  a concrete
$C^{*}$-$\cbaseosb$-$\cbasesc$-algebra
$(K,B,\beta,\gamma)$ , and a concrete
$C^{*}$-$\cbaseosc$-algebra $(L,C,\delta)$, we put
\begin{align*}
  \AfibreB \fibre{\gamma}{\frakK}{\delta} C := \big((\AfibreB)
  \fibre{\alpha \rt \gamma}{\frakK}{\delta} C\big) \cap \big(A
  \fibre{\alpha}{\frakH}{\beta \lt \delta} (B
  \fibre{\gamma}{\frakK}{\delta} C)\big);
\end{align*}
here, we canonically identify $ (\HfibreK) \rtensor{\alpha \rt
  \gamma}{\frakK}{\delta} L$ with $H \rtensor{\alpha}{\frakH}{\beta
  \lt \delta} (K \rtensor{\gamma}{\frakK}{\delta} L)$.

\begin{definition}
  Let $\cbasesb$ be a $C^{*}$-base and $(H,A,\alpha)$, $(K,B,\beta)$
  concrete $C^{*}$-$\cbasesb$-algebras. A {\em morphism} from
  $(H,A,\alpha)$ to $(K,B,\beta)$ is a $*$-homomorphism $\pi \colon A
  \to B$ such that $\beta = [{\cal
    L}^{\pi}(H_{\alpha},K_{\beta})\alpha]$, where ${\cal
    L}^{\pi}(H_{\alpha}, K_{\beta}) := \big\{ V \in {\cal
    L}(H_{\alpha},K_{\beta}) \,\big|\, \forall a \in A: \pi(a)V = V
  a\big\}$.
\end{definition}
 We denote the set of all morphisms from $(H,A,\alpha)$ to
 $(K,B,\beta)$ by $\Mor\big(A_{\alpha},B_{\beta})$.
In \cite[Definition 3.11]{timmer:cpmu}, we imposed an additional
condition on morphisms which is  automatically satisfied:
\begin{lemma}
  Let $\cbasesb$ be a $C^{*}$-base and let $\pi$ be a morphism of
  concrete $C^{*}$-$\cbasesb$-algebras $(H,A,\alpha)$ and
  $(K,B,\beta)$.  Then $\pi(a\rho_{\alpha}(b^{\dag})) =
  \pi(a)\rho_{\beta}(b^{\dag})$ for all $a \in A$ and $b^{\dag} \in
  \frakBo$.
\end{lemma}
\begin{proof}
  Let $a \in A$ and $b^{\dag} \in \frakBo$. Then for all $V \in {\cal
    L}^{\pi}(H_{\alpha},K_{\beta})$, $\xi \in \alpha$, $\zeta \in
  \frakH$,
  \begin{align*}
    \pi(a)\rho_{\beta}(b^{\dag}) V\xi\zeta = \pi(a) V\xi
    b^{\dag}\zeta = Va \rho_{\alpha}(b^{\dag}) \xi \zeta =
    \pi(a\rho_{\alpha}(b^{\dag})) V\xi\zeta.
  \end{align*}
  Since $\lnspan {\cal
    L}^{\pi}(H_{\alpha},K_{\beta})\alpha\frakH\rnspan  = \lnspan
  \beta\frakH\rnspan = K$, the claim follows.
\end{proof}
Let $\pi$ be a morphism of concrete $C^{*}$-$\cbasesb$-algebras
$(H,A,\alpha)$ and $(K,B,\beta)$.  Then $\lnspan \pi(A)K\rnspan = K$
\cite[Remark 3.12]{timmer:cpmu}.

Morphisms extend to multipliers as follows.  Let $(H,A,\alpha)$ be a
nondegenerate concrete  $C^{*}$-$\cbasesb$-algebra. Then clearly also
$(H,M(A),\alpha)$ is a concrete $C^{*}$-$\cbasesb$-algebra.
\begin{definition}
  Let $\cbasesb$ be a $C^{*}$-base and $(H,A,\alpha)$, $(K,B,\beta)$
  nondegenerate concrete $C^{*}$-$\cbasesb$-algebras. We call a
  morphism $\pi$ from $(H,A,\alpha)$ to $(K,M(B),\beta)$ {\em
    nondegenerate} if $\lnspan \pi(A)B\rnspan = B$.
\end{definition}
\begin{lemma}
  Let $\cbasesb$ be a $C^{*}$-base, let $(H,A,\alpha)$, $(K,B,\beta)$
  be nondegenerate concrete $C^{*}$-$\cbasesb$-algebras, and let $\pi$
  be a nondegenerate morphism from $(H,A,\alpha)$ to
  $(K,M(B),\beta)$. Then the unique strictly continuous extension
  $\tilde \pi  \colon M(A) \to 
  M(B)$ of $\pi$ is a morphism from $(H,M(A),\alpha)$ to $(K,M(B),\beta)$.
\end{lemma}
\begin{proof}
  We only need to show that ${\cal L}^{
    \pi}(H_{\alpha},K_{\beta})$  is contained in ${\cal L}^{\tilde
    \pi}(H_{\alpha},K_{\beta})$. But if $V \in {\cal
    L}^{\pi}(H_{\alpha},K_{\beta})$ and $T \in M(A)$, then
  $VTa\zeta =\pi(Ta)V\zeta =\tilde \pi(T)\zeta \pi(a)V\zeta=\tilde
  \pi(T)Va\zeta$ for all $a 
  \in A$ and $\zeta \in H$, and since $[AH] = H$, we can
  conclude $VT=\tilde \pi(T)V$.
\end{proof}

\begin{lemma}
  Let $(H,A,\alpha)$ be a nondegenerate concrete
  $C^{*}$-$\cbasesb$-algebra and $(K,B,\beta)$ a nondegenerate
  concrete $C^{*}$-$\cbaseosb$-algebra. Then $M(A)
  \hfibre{\alpha}{\beta}M(B) \subseteq M(A \hfibre{\alpha}{\beta} B)$.
\end{lemma}
\begin{proof}
  Let $T \in M(A) \hfibre{\alpha}{\beta}M(B)$. Then $T(A
  \hfibre{\alpha}{\beta} B) \subseteq A \hfibre{\alpha}{\beta} B$ because
  \begin{align*}
    T(A \hfibre{\alpha}{\beta} B) \kalpha{1} \subseteq
    T \lnspan \kalpha{1}B\rnspan \subseteq \lnspan \kalpha{1}
    M(B)B\rnspan \subseteq \lnspan \kalpha{1}B\rnspan
  \end{align*}
  and similarly 
  \begin{align*}
(A \hfibre{\alpha}{\beta} B)T^{*}\kalpha{1} &\subseteq \lnspan
  \kalpha{1}B\rnspan, &
    T(A \hfibre{\alpha}{\beta} B)\kbeta{2} &\subseteq \lnspan
    \kbeta{2}B\rnspan, \\
    (A \hfibre{\alpha}{\beta} B)T^{*}\kbeta{2} &\subseteq \lnspan
    \kbeta{2}B\rnspan. && \qedhere
  \end{align*}
\end{proof}

 Let $\phi$ be a morphism of nondegenerate concrete
 $C^{*}$-$\cbasesb$-algebras $(H,A,\alpha)$ and $(L,\gamma,C)$, and
 let $\psi$ be a morphism of nondegenerate concrete
 $C^{*}$-$\cbaseosb$-algebras $(K,B,\beta)$ and $(M,\delta,D)$.  Then
 there exists a unique $*$-homomorphism $\phi \ast \psi \colon
 \AfibreB \to C \fibre{\gamma}{\frakH}{\delta} D$ such that $(\phi
 \ast \psi)(T) \cdot (X \rtensorh Y) = (X \rtensorh Y) \cdot T$
 whenever $T \in \AfibreB$ and one of the following conditions holds:
 i) $X \in {\cal L}^{\phi}(H,L)$ and $Y \in {\cal
   L}^{\psi}(K_{\beta},D_{\delta})$ or ii) $X \in {\cal
   L}^{\phi}(H_{\alpha},L_{\gamma})$ and $Y \in {\cal L}^{\psi}(K,D)$.
 Moreover, let $\cbasesc$ be a $C^{*}$-base and assume that 
 $\AfibreB \subseteq {\cal L}(\HfibreK)$ is nondegenerate.  If
 $\alpha' \in \cfact(A_{\alpha};\cbasesc)$, $\gamma' \in
 \cfact(C_{\gamma};\cbasesc)$, $\phi \in \Mor(A_{\alpha'},
 C_{\gamma'})$, then $\phi \ast \psi \in \Mor
 \big((\AfibreB)_{(\alpha' \lt \beta)}, (C
 \fibre{\gamma}{\frakH}{\delta} D)_{(\gamma' \lt
   \delta)}\big)$. Similarly, if $\beta' \in
 \cfact(B_{\beta};\cbasesc)$, $\delta' \in
 \cfact(D_{\delta};\cbasesc)$, $\psi \in \Mor(B_{\beta'},
 D_{\delta'})$, then $\phi \ast \psi \in \Mor \big((\AfibreB)_{(\alpha
   \rt \beta')}, (C \fibre{\gamma}{\frakH}{\delta} D)_{(\gamma \rt
   \delta')}\big)$.

\begin{lemma} \label{lemma:ind-functorial}
  Let $H$ be a Hilbert space, $\alpha \in \cfact(H;\cbaseosb)$, and
  let $\rho$ be a morphism of concrete $C^{*}$-$\cbasesb$-algebras
  $(K,C,\gamma)$ and $(L,D,\delta)$. Then $\lnspan \kalpha{2} C
  \balpha{2}\rnspan \subseteq {\cal L}(K \rtensorca H)$ and $\lnspan
  \kalpha{2}D\balpha{2}\rnspan \subseteq {\cal L}(L
  \htensor{\delta}{\alpha} H)$ are $C^{*}$-algebras, and there exists
  a $*$-homomorphism $\Ind_{\alpha}(\rho) \colon \lnspan \kalpha{2} C
  \balpha{2}\rnspan \to \lnspan
  \kalpha{2}D\balpha{2}\rnspan$ such that for all $c \in C$ and
  $\xi,\xi' \in \alpha$,
  \begin{align*}
    \Ind_{\alpha}(\rho)
    \big(|\xi\rangle_{\leg{2}}c\langle\xi'|_{\leg{2}}\big) =
    |\xi\rangle_{\leg{2}}\rho(c)\langle\xi'|_{\leg{2}}.
  \end{align*}
\end{lemma}
\begin{proof}
  The space $\lnspan \kalpha{2} C \balpha{2}\rnspan$ is a
  $C^{*}$-algebra because 
  \begin{align*}
    \kalpha{2} C \balpha{2} (\kalpha{2} C \balpha{2})^{*}\subseteq
    \kalpha{2}C\rho_{\gamma}(\frakB)C\balpha{2} \subseteq
    \kalpha{2}C\balpha{2}.
  \end{align*}
  Likewise, $\lnspan \kalpha{2} D \balpha{2}\rnspan$ is a
  $C^{*}$-algebra. The existence of $\Ind_{\alpha}(\rho)$ follows as
  in the first part of the proof of \cite[Proposition
  3.13]{timmer:cpmu}.
\end{proof}
  A {\em concrete Hopf $C^{*}$-bimodule} is a tuple consisting of a
  $C^{*}$-base $\cbasesb$, a nondegenerate concrete
  $C^{*}$-$\cbasesb$-$\cbaseosb$-algebra $(H,A,\alpha,\beta)$, and a
  $*$-homomorphism $\Delta \colon A \to \AfibreA$ subject to the
  following conditions:
  \begin{enumerate}
  \item $\Delta \in \Mor\big(A_{\alpha}, (\AfibreA)_{\alpha \rt \alpha}\big)
    \cap \Mor\big(A_{\beta}, (\AfibreA)_{\beta \lt \beta}\big)$, and
  \item the following diagram commutes:
    \begin{align*}
      \xymatrix@C=10pt@R=10pt{
        A \ar[rrr]^{\Delta} \ar[dd]^{\Delta} &&& {\AfibreA} \ar[d]^{\Id
          \ast \Delta} \\ 
        &&& {A \fibre{\alpha}{\frakH}{\beta \lt \beta} (\AfibreA)}
        \ar@{^(->}[d] \\
        {\AfibreA} \ar[rr]^(0.35){\Delta\ast \Id} &&
        {(\AfibreA) \fibre{\alpha \rt \alpha}{\frakH}{\beta} A}
        \ar@{^(->}[r] &{\cal L}(H \fibreab H \fibreab H).
      }
    \end{align*}
  \end{enumerate}
  If $(\cbasesb,H,A,\alpha,\beta,\Delta)$ is a concrete Hopf
  $C^{*}$-bimodule, then the spaces $(\Delta \ast \Id)(\Delta(A))$ and $(\Id\ast
  \Delta)(\Delta(A))$ are contained in $A \fibreab A \fibreab A$.

  If $(\cbasesb,H,A,\alpha,\beta,\Delta)$ is a concrete Hopf
  $C^{*}$-bimodule and $\cbasesb$ and $H$ are understood, we shall
   denote this concrete Hopf
  $C^{*}$-bimodule briefly by $\Abimod$.

\paragraph{$C^{*}$-pseudo-multiplicative unitaries and the associated legs}
The notion of a $C^{*}$-pseudo-multiplicative unitary extends the
notion of a multiplicative unitary  \cite{baaj:2}, of a continuous
field of multiplicative unitaries \cite{blanchard}, and of a
pseudo-multiplicative unitary on $C^{*}$-modules
\cite{ouchi,timmermann:hopf}, and  is closely related to
pseudo-multiplicative unitaries on Hilbert spaces \cite{vallin:2}; see
\cite[Section 4.1]{timmer:cpmu}. The precise definition is as follows.

Let $H$ be a Hilbert space, $\cbasesb$ a $C^{*}$-base, $\alpha \in
\cfact(H;\cbasesb)$, $\hbeta \in \cfact(H;\cbaseosb)$, $\beta \in
\cfact(H;\cbaseosb)$ such that $\alpha,\beta,\hbeta$ are pairwise
compatible. Let $V \colon \Hsource \to \Hrange$ be a unitary such that
  \begin{gather} \label{eq:pmu-intertwine}
    \begin{aligned}
      V_{*}(\alpha \lt \alpha) &= \alpha \rt \alpha, &
      V_{*}(\hbeta \rt \beta) &= \hbeta \lt \beta, &
      V_{*}(\hbeta \rt \hbeta) &= \alpha \rt \hbeta, & 
      V_{*}(\beta \lt \alpha) &= \beta \lt \beta.
    \end{aligned}
  \end{gather}
  Then all operators in the following diagram are well-defined
  \cite[Lemma 4.1]{timmer:cpmu},
  \begin{gather} \label{eq:pmu-pentagon}
    \begin{gathered}
      \xymatrix@R=15pt{ {\Hone} \ar[r]^{ V
          \rtensorh \Id} \ar[d]_{\Id \rtensorh V} & {\Htwo}
        \ar[r]^{ \Id \rtensorh V}& { \Hthree,}
        \\
        {\Hfive} \ar[d]_{ \Id\rtensorh \Sigma} & & {\Hfour}
        \ar[u]_{ V \rtensorh \Id}
        \\
        {\Hfourlt} \ar[rr]^{ V \rtensorh \Id}&& {\Hfourrt} \ar[u]_{
          \Sigma_{\leg{23}}} }
    \end{gathered}
  \end{gather}
where $\Sigma_{\leg{23}}$ denotes the isomorphism
\begin{align*}
  \Hfourrt \cong (H {_{\rho_{\alpha}} \tl} \beta) {_{\rho_{\hbeta \lt
        \beta}} \tl} \alpha &\mycong (H {_{\rho_{\hbeta}} \tl} \alpha)
  {_{\rho_{\alpha \lt \alpha}} \tl} \beta \cong \Hfour, \\
  (\zeta \tl \xi) \tl \eta &\mapsto (\zeta \tl \eta) \tl \xi.
\end{align*}
We put $V_{\leg{12}}:=V \rtensorh \Id$, $V_{\leg{23}}:=\Id
\rtensorh V$, $V_{\leg{13}}:=(\Id \rtensorh
\Sigma)V_{\leg{12}}\Sigma_{\leg{23}}$. We call $V$  a
{\em $C^{*}$-pseudo-multiplicative unitary} if Diagram
\eqref{eq:pmu-pentagon} commutes, that is, if
$V_{\leg{12}}V_{\leg{13}}V_{\leg{23}}=V_{\leg{23}}V_{\leg{12}}$. In
that case, also $V^{op} :=\Sigma V^{*} \Sigma \colon H
\htensor{\beta}{\alpha} H \to H \htensor{\alpha}{\hbeta} H$ is a
$C^{*}$-pseudo-multiplica\-tive unitary, called the {\em opposite} of
$V$. 

Let $V \colon \Hsource \to \Hrange$ be a $C^{*}$-pseudo-multiplicative
unitary.  Then the spaces
\begin{align*}
\hA:=  \hA(V) &:= \lnspan \bbeta{2} V \kalpha{2}\rnspan \subseteq {\cal
    L}(H), & A:= A(V) &:= \lnspan\balpha{1} V \khbeta{1}\rnspan \subseteq
  {\cal L}(H)
\end{align*}
 satisfy
  \begin{gather*}
\lnspan \hA \hA\rnspan=    \lnspan \hA \rho_{\hbeta}(\frakB)\rnspan = \lnspan \rho_{\hbeta}(\frakB)
    \hA\rnspan = \lnspan \hA \rho_{\alpha}(\frakBo)\rnspan = \lnspan
    \rho_{\alpha}(\frakBo) \hA\rnspan = \hA \subseteq {\cal
      L}(H_{\beta}),\\ \lnspan A A \rnspan = \lnspan A
    \rho_{\beta}(\frakB)\rnspan = \lnspan \rho_{\beta}(\frakB) A\rnspan =
    \lnspan A \rho_{\alpha}(\frakBo)\rnspan = \lnspan
    \rho_{\alpha}(\frakBo) A\rnspan = A \subseteq {\cal L}(H_{\hbeta}).
  \end{gather*}
  Define maps
\begin{align*}
 \widehat{\Delta}_{V} &\colon \hA \to
  {\cal L}\big(\Hsource\big), \ y \mapsto V^{*}(1 \rtensorh
  y)V,  & \Delta_{V} &\colon A \to {\cal L}\big(\Hrange), \ z
  \mapsto V(z \rtensorh 1)V^{*}.
\end{align*}
We call $V$ {\em well-behaved} if the tuples
$\big(\cbaseosb,H,\hA,\hbeta,\alpha,\hDelta_{V}\big)$ and
$\big(\cbasesb, H, A,\alpha,\beta,\Delta_{V}\big)$ are concrete Hopf
$C^{*}$-bimodules.  We call $V$ {\em regular} if the subspace $\lnspan
\balpha{1} V \kalpha{2} \rnspan \subseteq {\cal L}(H)$ is equal to
$\lnspan \alpha \alpha^{*} \rnspan$. 
If $V$ is regular, then it is well-behaved
\cite[Theorem 4.14]{timmer:cpmu}.

\section{Weak $C^{*}$-pseudo-Kac systems}

Reduced crossed products for coactions of Hopf $C^{*}$-algebras can
conveniently be constructed in the framework of Kac systems
\cite{baaj:2} or, more generally, of weak Kac systems
\cite{vergnioux}. To adapt the construction to coactions of concrete
Hopf $C^{*}$-bimodules, we generalize the notion of a weak Kac system
as follows.

Let $H$ be a Hilbert space. Recall that for each $C^{*}$-base
$\cbasesc$ and each $C^{*}$-factorization $\gamma \in
\cfact(H;\cbasesc)$, $\delta \in \cfact(H;\cbaseosc)$, there exists a
flip map
\begin{align*}
  \Sigma \colon H \rtensor{\gamma}{\frakK}{\delta} H = \gamma \tr
  \frakK \tl \delta \to \delta \tr \frakK \tl \gamma = H
  \rtensor{\delta}{\frakK}{\gamma} H, \quad \eta \tr \zeta \tl \xi
  \mapsto \xi \tr \zeta \tl \eta.
\end{align*}

Let $\cbasesb$ be a $C^{*}$-base, $\alpha,\halpha \in
\cfact(H;\cbasesb)$, $\beta,\hbeta \in \cfact(H;\cbaseosb)$, and let
$U \colon H \to H$ be a symmetry, that is, a self-adjoint
unitary. Assume that $U\alpha=\halpha$ and $U\beta=\hbeta$; then also
$U\halpha=\alpha$ and $U\hbeta=\beta$.  For each $T \in {\cal
  L}(\Hsource, \Hrange)$, put
\begin{align*}
      \widecheck{T} &:=\Sigma (1 \rtensorh U)T(1 \rtensorh U)\Sigma \colon
      \checkHsource \xrightarrow{U_{\leg{2}}\Sigma} \Hsource
      \xrightarrow{T} \Hrange \xrightarrow{\Sigma U_{\leg{2}}}
      \checkHrange, \\ \widehat{T} &:= \Sigma (U \rtensorh
      1)T(U \rtensorh 1)\Sigma \colon \hatHsource
      \xrightarrow{U_{\leg{1}}\Sigma} \Hsource \xrightarrow{T}
      \Hrange \xrightarrow{\Sigma U_{\leg{1}}} \hatHrange.
\end{align*}
Switching between the $C^{*}$-bases $\cbasesb$ and $\cbaseosb$ and
relabeling the $C^{*}$-factorizations $\alpha,\halpha,\beta,\hbeta$
suitably, we can iterate the maps $T \mapsto \widecheck{T}$ and $T
\mapsto \widehat{T}$.  The two relations
$\Sigma(1 \rtensorh U) \Sigma(1\rtensorh U) = U \rtensorh U =
  \Sigma(U \rtensorh 1) \Sigma(U \rtensorh 1)$ and 
  $\Sigma(1 \rtensorh U) \Sigma (U \rtensorh 1) = \Id$ imply
\begin{align} \label{eq:balanced-iterate}
  \widecheck{\widecheck{T}} &= \Ad_{(U \rtensorh U)}(T) =
  \widehat{\widehat{T}}, & \widecheck{\widecheck{\widecheck{T}}} &=
  \widehat{T}, & \widecheck{T} &= \widehat{\widehat{\widehat{T}}} &&
  \text{for all } T \in {\cal L}(\Hsource,\Hrange).
\end{align}
\begin{definition} 
  A {\em balanced $C^{*}$-pseudo-multiplicative unitary}
  $(\alpha,\halpha,\beta,\hbeta,U,V)$ consists of
  \begin{itemize}
  \item $C^{*}$-factorizations $\alpha,\halpha \in \cfact(H;\cbasesb)$
    and $\beta,\hbeta \in \cfact(H;\cbaseosb)$,
  \item a symmetry $U \colon H \to H$, and
  \item a $C^{*}$-pseudo-multiplicative unitary $V \colon \Hsource \to
    \Hrange$
  \end{itemize}
 satisfying the following conditions:
  \begin{enumerate}
  \item $\alpha,\halpha,\beta,\hbeta$ are pairwise compatible,
  \item $U\alpha=\halpha$ and $U\beta=\hbeta$,
  \item $\checkV$ and $\hatV$  are $C^{*}$-pseudo-multiplicative
    unitaries.  
  \end{enumerate}
\end{definition}
\begin{remarks}\label{remarks:balanced}
  \begin{enumerate}
  \item Since $\hatV = (U \rtensorh U)\checkV (U
    \rtensorh U)$, the unitary $\hatV$ is $C^{*}$-pseudo-multiplicative
    if and only if $\checkV$ is $C^{*}$-pseudo-multiplicative.
  \item If $(\alpha,\halpha,\beta,\hbeta,U,V)$ is a balanced
    $C^{*}$-pseudo-multiplicative unitary, then also the tuples
    $(\hbeta,\beta,\alpha,\halpha,U,\checkV)$,
    $(\halpha,\alpha,\hbeta,\beta,U, \Ad_{(U \rtensorh U)}(V))$, and
    $(\beta,\hbeta,\halpha,\alpha,U,\hatV)$ are balanced
    $C^{*}$-pseudo-multiplicative unitaries. This follows easily from
    \eqref{eq:balanced-iterate}. Moreover, in that case
    $(\alpha,\halpha,\hbeta,\beta,U,V^{op})$ is a balanced
    $C^{*}$-pseudo-multiplicative unitary as well, as can be seen from
    the relation
    \begin{gather} \label{eq:balanced-op}
      \widecheck{(V^{op})} = \Sigma U_{\leg{2}} (\Sigma V^{*} \Sigma)
      U_{\leg{2}} \Sigma = \Sigma (\Sigma U_{\leg{1}} V U_{\leg{1}}
      \Sigma)^{*} \Sigma = (\hatV)^{op}.
   \end{gather}
 \item If $(\alpha,\halpha,\beta,\hbeta,U,V)$ is a balanced
   $C^{*}$-pseudo-multiplicative unitary, the relations
   \eqref{eq:pmu-intertwine} for the unitaries $\checkV \colon
   \checkHsource \to \checkHrange$ and $\hatV \colon \hatHsource \to
   \hatHrange$ read as follows:
   \begin{align*}
      \checkV_{*}(\hbeta \lt \hbeta) &= \hbeta \rt \hbeta, &
      \checkV_{*}(\halpha \rt \alpha) &= \halpha \lt \alpha, &
      \checkV_{*}(\halpha \rt \halpha) &= \hbeta \rt \halpha, & 
      \checkV_{*}(\alpha \lt \hbeta) &= \alpha \lt \alpha, \\
      \hatV_{*}(\beta \lt \beta) &= \beta \rt \beta, &
      \hatV_{*}(\alpha \rt \halpha) &= \alpha \lt \halpha, &
      \hatV_{*}(\alpha \rt \alpha) &= \beta \rt \alpha, & 
      \hatV_{*}(\halpha \lt \beta) &= \halpha \lt \halpha.
   \end{align*}
   These relations furthermore imply
   \begin{align*}
     V_{*}(\hbeta \rt \halpha) &= \alpha \rt \halpha, & 
     \checkV_{*}(\halpha \rt \beta) &= \hbeta \rt \beta, &
     \hatV_{*}(\alpha \rt \hbeta) &= \beta \rt \hbeta, \\
     V_{*}(\halpha \lt \alpha) &= \halpha \lt \beta, &
     \checkV_{*}(\beta \lt \hbeta) &= \beta \lt \alpha, &
     \hatV_{*}(\hbeta \lt \beta) &= \hbeta \lt \halpha.
   \end{align*}
 \item If $(\alpha,\halpha,\beta,\hbeta,U,V)$ is a balanced
   $C^{*}$-pseudo-multiplicative unitary, then $\hA(V),A(V) \subseteq {\cal
     L}(H_{\halpha})$ because
   \begin{gather} \label{eq:ha-halpha}
     \lnspan \hA(V) \halpha \rnspan = \lnspan
     \bbeta{2}V\kalpha{2}\halpha\rnspan = \lnspan \bbeta{2} \kbeta{2}
     \halpha\rnspan = \lnspan \rho_{\alpha}(\frakBo)\halpha\rnspan =
     \halpha
   \end{gather}
   and similarly $\lnspan A(V)\halpha\rnspan= \lnspan
   \balpha{1}V\khbeta{1}\halpha\rnspan = \lnspan
   \balpha{1}\kalpha{1}\halpha\rnspan = \lnspan
   \rho_{\beta}(\frakB)\halpha\rnspan = \halpha$.
  \end{enumerate}
\end{remarks}
The auxiliary unitaries $\checkV$ and $\hatV$ defined above allow us
to treat the right and the left leg of $V$, respectively, as the left
or the right leg of some $C^{*}$-pseudo-multiplicative unitary:
\begin{proposition} \label{proposition:balanced-legs}
  Let $(\alpha,\halpha,\beta,\hbeta,U,V)$ be a  balanced
  $C^{*}$-pseudo-multiplicative unitary.  Then
  \begin{gather*}
    \hA(\checkV) =  \Ad_{U}(A(V)), \quad
    \hDelta_{\checkV} = \Ad_{(U \rtensorh U)} \circ \Delta_{V} \circ
    \Ad_{U}, \qquad
    A(\checkV) = \hA(V), \quad
    \Delta_{\checkV} = \hDelta_{V}, \\
    A(\hatV) = \Ad_{U}(\hA(V)), \quad
    \Delta_{\hatV} = \Ad_{(U \rtensorh U)} \circ \hDelta_{V} \circ
    \Ad_{U}, \qquad
    \hA(\hatV) = A(V), \quad
    \hDelta_{\hatV} = \Delta_{V}.
  \end{gather*}
  In particular, $\checkV$ and $\hatV$ are well-behaved if $V$ is
  well-behaved.  
\end{proposition}
For the proof, we need the following lemma:
\begin{lemma} \label{lemma:balanced}
  The following diagrams commute:
  \begin{gather} \label{eq:balanced-rel1} \smalldiagram
    \xymatrix@C=35pt@R=15pt{
      {(\checkHsource) \htensor{\hbeta \lt \hbeta}{\alpha} H}
      \ar[r]^{V_{\leg{13}}} \ar[d]_{\checkV_{\leg{12}}} &
      {(\checkHsource) \htensor{\alpha \lt \hbeta}{\beta} H}
      \ar[r]^{\checkV_{\leg{12}}} &
      {(\Hsource) \htensor{\alpha \lt \alpha}{\beta} H} \\
      {\Hone} \ar[rr]^{V_{\leg{23}}}
      && {H \htensor{\hbeta}{\alpha\rt \alpha} (\Hrange)} \ar[u]_{V_{\leg{13}}}
    }
    \shortintertext{and} \label{eq:balanced-rel2}
    \xymatrix@C=35pt@R=15pt{
      {H \htensor{\hbeta}{\alpha\rt \alpha} (\hatHsource)}
      \ar[r]^{\hatV_{\leg{23}}} \ar[d]_{V_{\leg{13}}} &
      {H \htensor{\hbeta}{\beta \rt \alpha} (\hatHrange)}
      \ar[r]^{V_{\leg{13}}} &
      {H \htensor{\alpha}{\beta \rt \beta} (\hatHrange).} \\
      {(\Hsource) \htensor{\alpha \lt \alpha}{\beta} H}
      \ar[rr]^{V_{\leg{12}}} &&
      {\Hthree} \ar[u]_{\hatV_{\leg{23}}}
    }
  \end{gather}
\end{lemma}
\begin{proof}
  Let us prove that diagram \eqref{eq:balanced-rel1} commutes. Put
  $W:= \Sigma V \Sigma$. We insert the relation $\checkV = U_{\leg{1}}
  W U_{\leg{1}}$ into the pentagon equation $\checkV_{\leg{12}}
  \checkV_{\leg{13}} \checkV_{\leg{23}} = \checkV_{\leg{23}}
  \checkV_{\leg{12}}$ and obtain
  \begin{align*}
    U_{\leg{1}} W_{\leg{12}} U_{1} \cdot U_{1} W_{\leg{13}} U_{1}
    \cdot \checkV_{\leg{23}} = \checkV_{\leg{23}} \cdot U_{\leg{1}}
    W_{\leg{12}} U_{\leg{1}}. 
  \end{align*}
  Since $U_{\leg{1}}$ commutes with $\checkV_{\leg{23}}$, we can
  cancel $U_{\leg{1}}$ everywhere in the equation above and find
  $W_{\leg{12}} W_{\leg{13}} \checkV_{\leg{23}} = \checkV_{\leg{23}}
  W_{\leg{12}}$. Now, we conjugate both sides of this equation by the
  automorphism $\Sigma_{\leg{23}}\Sigma_{\leg{12}}$, which amounts to
  renumbering the legs of the operators according to the permutation
  $(1,2,3) \mapsto (2,3,1)$, and find $V_{\leg{13}}V_{\leg{23}}
  \checkV_{\leg{12}} = \checkV_{\leg{12}} V_{\leg{13}}$. If we retrace
  the derivation of this equation in diagrammatic form, we obtain  
  diagram \eqref{eq:balanced-rel1}.

  A similar argument shows that diagram \eqref{eq:balanced-rel2} commutes.
\end{proof}

\begin{proof}[Proof of Proposition \ref{proposition:balanced-legs}]
  Inserting the relations
  \begin{align*}
  \checkV=\Sigma
  U_{\leg{2}}VU_{\leg{2}}\Sigma \colon \checkHsource \to \checkHrange,
  &&
  \hatV=\Sigma U_{\leg{1}} V
  U_{\leg{1}} \Sigma \colon \hatHsource \to \hatHrange
  \end{align*}
  into the definition of $A(\checkV)$ and $\hA(\hatV)$, respectively, we find
  \begin{align*}
    A(\checkV) &= 
    \lnspan \bhbeta{1} \Sigma U_{\leg{2}} V U_{\leg{2}} \Sigma
    \khalpha{1} \rnspan = \lnspan \langle U \hbeta|_{\leg{2}} V |
    U\halpha\rangle_{\leg{2}} \rnspan = \lnspan \bbeta{2} V
    \kalpha{2}\rnspan = \hA(V), \\
    \hA(\hatV) &= \lnspan \bhalpha{2} \Sigma U_{\leg{1}} V U_{\leg{1}}
    \Sigma \kbeta{2}\rnspan = \lnspan \langle U\halpha|_{\leg{1}} V |
    U \beta\rangle_{\leg{1}}\rnspan = \lnspan \balpha{1} V \khbeta{1}
    \rnspan = A(V).
  \end{align*}

  To prove $\Delta_{\checkV} = \hDelta_{V}$, consider an element
  $\ha=\langle \xi'|_{\leg{2}}V|\xi\rangle_{\leg{2}} \in \hA(V)$,
   where $\xi' \in \beta$, $\xi \in \alpha$.  By definition,
   $\Delta_{\checkV}(\ha) = \checkV(\ha \rtensorh 1)\checkV^{*}$. The
 commutative diagram 
 \begin{gather*}\smalldiagram \hspace{-3pt}
    \xymatrix@R=20pt@C=18pt{
      {\checkHrange} \ar[r]^{\checkV^{*}}
      \ar[d]^{|\xi\rangle_{\leg{3}}} & {\checkHsource}
      \ar[r]^{\ha \rtensorh 1} \ar[d]^{|\xi\rangle_{\leg{3}}} &
      {\checkHsource} \ar[r]^{\checkV} &
      {\checkHrange,}  \\
      {\Hone}
      \ar[r]^(0.47){\checkV^{*}_{\leg{12}}}
          & {(\checkHsource)
        \htensor{\hbeta \lt \hbeta}{\alpha} H} \ar[r]^{V_{\leg{13}}}
      & {(\checkHsource) \htensor{\alpha \lt \hbeta}{\beta} H}
      \ar[r]^{\checkV_{\leg{12}}} \ar[u]^{\langle\xi'|_{\leg{3}}} &
      {(\Hsource) \htensor{\alpha \lt \alpha}{\beta} H}
      \ar[u]^{\langle\xi'|_{\leg{3}}}
    }
  \end{gather*}
 diagram  \eqref{eq:balanced-rel1} and the pentagon diagram
  \eqref{eq:pmu-pentagon} imply
  \begin{align*}
     \checkV(\ha \rtensorh 1)\checkV^{*} &= \langle
    \xi'|_{\leg{3}} \checkV_{\leg{12}}V_{\leg{13}}
    \checkV_{\leg{12}}^{*} |\xi\rangle_{\leg{3}} \\ &= \langle
    \xi'|_{\leg{3}} V_{\leg{13}} V_{\leg{23}}|\xi\rangle_{\leg{3}} =
    \langle \xi'|_{\leg{3}} V_{\leg{12}}^{*} V_{\leg{23}}V_{\leg{12}}
    |\xi\rangle_{\leg{3}}.
  \end{align*}
  The following diagram shows that the expression above is equal to
  $V^{*}(1 \rtensorh \ha)V=\hDelta_{V}(\ha)$:
  \begin{gather*}\smalldiagram \hspace{-3pt}
    \xymatrix@R=20pt@C=18pt{
      {\Hsource} \ar[r]^{V} \ar[d]^{|\xi\rangle_{\leg{3}}} &
      {\Hrange} \ar[r]^{1 \rtensorh \ha} \ar[d]^{|\xi\rangle_{\leg{3}}} &
      {\Hrange} \ar[r]^{V^{*}} &
      {\Hsource.}\\
      {\Hone} \ar[r]^{V_{\leg{12}}}
      & {\Htwo} \ar[r]^{V_{\leg{23}}}
      & {\Hthree} \ar[r]^(0.47){V_{\leg{12}}^{*}} \ar[u]^{\langle\xi'|_{\leg{3}}}&
      {\Hfour} \ar[u]^{\langle\xi'|_{\leg{3}}}
    }
  \end{gather*}
  Since elements of the form like $\ha$  are dense in $\hA(V)$, we can
  conclude $\Delta_{\checkV}=\hDelta_{\checkV}$.

  A similar argument shows that $\hDelta_{\hatV}=\Delta_{V}$; here, we
  have to use diagram \eqref{eq:balanced-rel2}.

  The remaining equations in Proposition
  \ref{proposition:balanced-legs} follow  from the relation
  $\hatV = \Ad_{(U \rtensorh U)}(\checkV)$ and those equations that we
  have proved already.
\end{proof}
The definition of a weak  $C^{*}$-pseudo-Kac system involves the
following conditions:
\begin{lemma} \label{lemma:weak-kac}
 Let $(\alpha,\halpha,\beta,\hbeta,U,V)$ be a balanced
  $C^{*}$-pseudo-multiplicative unitary.
  \begin{enumerate}
  \item The following conditions are equivalent:
    \begin{enumerate}
    \item The following diagram commutes: \raisebox{-3pt}{
        $\smalldiagram\xymatrix@R=20pt{ {\hatHsource
            \htensor{\hbeta}{\alpha} H} \ar[r]^{\hatV_{\leg{12}}}
          \ar[d]^{V_{\leg{23}}} &
          {\hatHrange \htensor{\hbeta}{\alpha} H} \ar[d]^{V_{\leg{23}}} \\
          {\hatHsource \htensor{\alpha}{\beta} H}
          \ar[r]^{\hatV_{\leg{12}}} & {\hatHrange
            \htensor{\alpha}{\beta} H.} }$}
    \item  $(1 \rtensorh \ha)\hatV=\hatV(1 \rtensorh \ha)$ in ${\cal
        L}(\hatHsource, \hatHrange)$ for each $\ha \in \hA(V)$.
    \item $(\Ad_{U}(\ha) \rtensorh 1)V = V(\Ad_{U}(\ha) \rtensorh 1)$
      in ${\cal L}(\Hsource, \Hrange)$ for each $\ha \in
      \hA(V)$.
   \item $\hA(V)$ and $\Ad_{U}(\hA(V))$ commute.
    \end{enumerate}
  \item The following conditions are equivalent:
    \begin{enumerate}
    \item The following diagram commutes:
      \raisebox{-3pt}{
        $\smalldiagram\xymatrix@R=20pt{ {\Hsource \htensor{\halpha}{\hbeta} H}
          \ar[r]^{V_{\leg{12}}} \ar[d]^{\checkV_{\leg{23}}}
           & {\Hrange \htensor{\halpha}{\hbeta} H}
          \ar[d]^{\checkV_{\leg{23}}} \\
          {\Hsource \htensor{\hbeta}{\alpha} H} \ar[r]^{V_{\leg{12}}}
          & {\Hrange \htensor{\hbeta}{\alpha} H.}  }$
      }
    \item  $(a \rtensorh 1)\checkV=\checkV(a \rtensorh 1)$ in ${\cal
        L}(\checkHsource, \checkHrange)$ for each $a \in A(V)$.
    \item $(1 \rtensorh \Ad_{U}(a))V = V(1 \rtensorh \Ad_{U}(a))$ in
      ${\cal L}(\Hsource, \Hrange)$ for each $a \in A(V)$.
   \item $A(V)$ and $\Ad_{U}(A(V))$ commute.
    \end{enumerate}
  \end{enumerate}
\end{lemma}
\begin{proof}
  In i),  conditions (a) and (b) are equivalent because
  \begin{align*}
    (a) \ &\Leftrightarrow  \ \forall \xi \in \alpha, \xi' \in \beta:
    \langle \xi'|_{\leg{3}}
    V_{\leg{23}}\hatV_{\leg{12}}|\xi\rangle_{\leg{3}} = \langle
    \xi'|_{\leg{3}}
    \hatV_{\leg{12}}    V_{\leg{23}}|\xi\rangle_{\leg{3}}      \\
    &\Leftrightarrow \ \forall \xi \in \alpha, \xi' \in \beta: (1
    \rtensorh \ha) \hatV = \hatV (1 \rtensorh \ha), \text{ where } \ha
    = \langle\xi'|_{\leg{2}}V|\xi\rangle_{\leg{2}} \ \Leftrightarrow
    \ (b).
  \end{align*}
  The remaining equivalences follow similarly.
\end{proof}
\begin{definition}
  We call a balanced $C^{*}$-pseudo-multiplicative unitary
  $(\alpha,\halpha,\beta,\hbeta,U,V)$ a {\em weak $C^{*}$-pseudo-Kac
    system} if $V$ is well-behaved and if the equivalent conditions in
  Lemma \ref{lemma:weak-kac} hold.
\end{definition}
The notion of a weak $C^{*}$-pseudo-Kac system is symmetric in the
following sense:
\begin{proposition}
  Let $(\alpha,\halpha,\beta,\hbeta,U,V)$ be a weak $C^{*}$-pseudo-Kac
  system.  Then also the following tuples are weak $C^{*}$-pseudo-Kac
  systems:
  \begin{gather} \label{eq:weak-kac-symmetry}
    (\hbeta,\beta,\alpha,\halpha,U,\checkV), \quad
    (\beta,\hbeta,\halpha,\alpha,U,\hatV), \quad
    (\alpha,\halpha,\hbeta,\beta,U,V^{op}).
\end{gather}
\end{proposition}
\begin{proof}
  The $C^{*}$-pseudo-multiplicative unitaries $\checkV$, $\hatV$,
  $V^{op}$ are well-behaved by Proposition
  \ref{proposition:balanced-legs} and \cite[Lemma 4.4]{timmer:cpmu},
  respectively, and the tuples \eqref{eq:weak-kac-symmetry} are
  balanced $C^{*}$-pseudo-multiplica\-tive unitaries by Remark
  \ref{remarks:balanced} ii).  Using Proposition
  \ref{proposition:balanced-legs} and \cite[Lemma 4.4]{timmer:cpmu},
  one easily checks that these tuples satisfy the conditions i)(d) and
  ii)(d) of Lemma \ref{lemma:weak-kac}.
\end{proof}
\begin{definition}
  Given a weak $C^{*}$-pseudo-Kac system
  $(\alpha,\halpha,\beta,\hbeta,U,V)$, we call the tuples
  \eqref{eq:weak-kac-symmetry} its {\em predual}, its {\em dual}, and
  its {\em opposite}, respectively.
\end{definition}

\section{Coactions and reduced crossed products}

\paragraph{Coactions of concrete Hopf $C^{*}$-bimodules}
\begin{definition}
  Let $(\cbasesb,H,A,\alpha,\beta,\Delta)$ be a concrete Hopf
  $C^{*}$-bimodule and $(K,C,\gamma)$ a nondegenerate concrete
  $C^{*}$-$\cbasesb$-algebra. Then a {\em coaction} of
  $(\cbasesb,H,A,\alpha,\beta,\Delta)$ on $(K,C,\gamma)$ is a morphism
  $\delta_{C}$ from $(K,C,\gamma)$ to $(K \rtensorcb H, C
  \hfibrecb M(A), \gamma \rt \alpha)$ that makes the
  following diagram commute:
    \begin{align*}
      \xymatrix@C=20pt@R=10pt{
        C \ar[rrr]^{\delta_{C}} \ar[ddd]^{\delta_{C}} & &&
        {C \hfibrecb M(A)} \ar[dd]^{\delta_{C} \ast \Id_{M(A)}} \\
        && & \\
        && & {(C \hfibrecb M(A)) \hfibre{\gamma \rt
            \alpha}{\beta} M(A)} \ar@{^(->}[d] \\ 
        {C \hfibrecb M(A)} \ar[rr]^(0.38){\Id_{C} \ast \Delta} &&
        {C \hfibre{\gamma}{\beta \lt \beta} M(A \hfibre{\alpha}{\beta} A)}
        \ar@{^(->}[r] &
        {\cal L}(K \htensor{\gamma}{\beta} \Hrange).
      }
    \end{align*}
    We also refer to the tuple $(K,C,\gamma,\delta_{C})$ as a
    coaction.  We call such a coaction {\em fine} if
    $\delta_{C}$ is injective and $\lnspan \delta_{C}(C)\kbeta{2}\rnspan =
    \lnspan \kbeta{2}C\rnspan$ as subsets of ${\cal L}(K,K
    \rtensorcb H)$.

    A {\em covariant morphism} between coactions
    $(K,C,\gamma,\delta_{C})$ and $(L,D,\delta,\delta_{D})$ is a
    morphism $\rho$
    from $(K,C,\gamma)$ to $(L,M(D),\delta)$ that makes the following
    diagram commute:
    \begin{align*}
      \xymatrix@C=35pt{
        C \ar[r]^{\rho} \ar[d]^{\delta_{C}} & M(D) \ar[d]^{\delta_{D}} \\
        {C \hfibrecb M(A)} \ar[r]^{\rho \ast \Id_{M(A)}} &
        {M(D \hfibre{\delta}{\beta} A)}.
      }
    \end{align*}
\end{definition}
\begin{remarks} \label{remarks:coaction}
  \begin{enumerate}
  \item Note that by \cite[Remark 3.12]{timmer:cpmu}, the
    $C^{*}$-algebra $\delta_{C}(C)$ and hence also the $C^{*}$-algebra
    $C \hfibrecb A$ is nondegenerate.
  \item Evidently, the class of all coactions of a fixed concrete
    Hopf-$C^{*}$-bimodule forms a category with respect to covariant
    morphisms.
  \item For every concrete Hopf $C^{*}$-bimodule
    $(\cbasesb,H,A,\alpha,\beta,\Delta)$, the triple 
    $(H,A,\alpha,\Delta)$ is a coaction.
  \end{enumerate}
\end{remarks}

We shall study coactions of concrete Hopf $C^{*}$-bimodules in a
separate article.

\paragraph{Reduced crossed products for coactions of $\big(\cbasesb,H, A,\alpha,\beta,\Delta_{V}\big)$}

Till the end of this section, we fix a weak $C^{*}$-pseudo-Kac system
$(\alpha,\halpha,\beta,\hbeta,U,V)$. To shorten the notation, we put
\begin{align*}
    \hA&:=\hA(V), &\hDelta&:=\hDelta_{V}, & A&:=A(V),
    &\Delta:=\Delta_{V},
\end{align*}
and write $\hAbimod$ and
$\Abimod$ for the concrete Hopf
$C^{*}$-bimodules $\big(\cbaseosb,H,\hA,\hbeta,\alpha,\hDelta\big)$ and
$\big(\cbasesb,H, A,\alpha,\beta,\Delta\big)$, respectively.

First, we define reduced crossed products for coactions of  $\Abimod$:
\begin{definition} \label{definition:rcp} Let
  $(K,C,\gamma,\delta_{C})$ be a coaction of $\Abimod$. The
  associated {\em reduced crossed product} is the $C^{*}$-subalgebra
  $C \rtimes_{\delta_{C},r} \hA \subseteq {\cal L}(K \rtensorcb H)$
  generated by
  \begin{align} \label{eq:rcp-set}
    \delta(C)(1 \rtensorh \hA) \subseteq {\cal L}(K
    \rtensorcb H).
  \end{align}
  If $\delta_{C}$ is understood, we shortly write $C \rtimes_{r}
  \hA$ for $C \rtimes_{\delta_{C},r} \hA$.
\end{definition}
\begin{remark} \label{remark:rcp-set} In the situation above, $C
  \rtimes_{r} \hA \subseteq C \hfibrecb {\cal L}(H)$,
  as can be seen from the inclusions $\delta_{C}(C) \subseteq C \hfibrecb
  M(A)$ and
  \begin{align*}
    (C \hfibrecb M(A))(1 \rtensorh\hA)|\gamma\rangle_{\leg{1}} &=
    (C \hfibrecb M(A)) |\gamma\rangle_{\leg{1}} \hA \subseteq
    \lnspan    |\gamma\rangle_{\leg{1}} M(A)\hA \rnspan, \\
    (C \hfibrecb M(A))(1 \rtensorh \hA)\kbeta{2} &\subseteq
    \lnspan (C \hfibrecb M(A)) \kbeta{2} \rnspan \subseteq \lnspan
    \kbeta{2}C\rnspan;
  \end{align*}
  here, we used the inclusion $\hA\beta \subseteq \beta$
  \cite[Lemma 4.5]{timmer:cpmu}.
\end{remark}
Frequently, it is useful to know that  the set  
\eqref{eq:rcp-set} is linearly dense in $C \rtimes_{r} \hA$:
\begin{proposition} \label{proposition:rcp-set} Let
  $(K,C,\gamma,\delta_{C})$ be a coaction of $\Abimod$.  Then
  \begin{gather*}
    C \rtimes_{r} \hA = \lnspan \delta_{C}(C)(1 \rtensorh
    \hA)\rnspan.
  \end{gather*}
\end{proposition}
\begin{proof}
  We only need to prove $\lnspan (1 \rtensorh \hA) \delta_{C}(C)
  \rnspan \subseteq \lnspan \delta_{C}(C)(1 \rtensorh \hA)\rnspan$. By
  definition of $\hA$, 
  \begin{align*}
    \lnspan (1 \rtensorh \hA) \delta_{C}(C) \rnspan = \lnspan
    \bbeta{3}(1 \rtensorh V)\kalpha{3}\delta_{C}(C)\rnspan =
    \lnspan \bbeta{3}(1 \rtensorh V)(\delta_{C}(C) \rtensorh
    1)\kalpha{3}\rnspan. 
  \end{align*}
  Note that $\delta_{C}(C) \rtensorh 1 \subseteq {\cal L}(K \rtensorcb
  \Hsource)$ is well-defined because $\delta_{C}(C) \subseteq C
  \hfibrecb M(A)$ commutes with $1 \rtensorh \rho_{\hbeta}(\frakB)$ by
  \cite[Lemma 3.8]{timmer:cpmu}. Put  $\delta^{(2)}_{C}:=(\Id \ast \Delta)\circ
  \delta$. By definition of $\Delta$,
  \begin{align*}
    \lnspan \bbeta{3}(1 \rtensorh V)(\delta_{C}(C) \rtensorh
    1)\kalpha{3}\rnspan = \lnspan \bbeta{3}(\delta_{C}^{(2)}(C))(1
    \rtensorh V)\kalpha{3}\rnspan.
  \end{align*}
  Since $\delta^{(2)}_{C}(C) = (\Id \ast \Delta)(\delta_{C}(C)) =
  (\delta_{C} \ast \Id)(\delta_{C}(C)) \subseteq \delta_{C}(C)
  \hfibre{\gamma \rt \alpha}{\beta} M(A)$,
  \begin{gather*}
    \lnspan \bbeta{3}(\delta_{C}^{(2)}(C))(1
    \rtensorh V)\kalpha{3}\rnspan \subseteq \lnspan
    \delta_{C}(C)\bbeta{3}(1 \rtensorh V)\kalpha{3}\rnspan = \lnspan
    \delta_{C}(C)(1 \rtensorh \hA)\rnspan. \qedhere
  \end{gather*}
\end{proof}
\begin{corollary}
  Let $(K,C,\gamma,\delta_{C})$ be a coaction of $\Abimod$.  Then
  $\delta_{C}(C)$ and $1 \rtensorh \hA$ are nondegenerate
  $C^{*}$-subalgebras of $M(C \rtimes_{r} \hA)$, and the maps $c
  \mapsto \delta_{C}(c)$ and $\ha \mapsto 1 \rtensorh \ha$ extend to
  $*$-homomorphisms $M(C) \to M(C \rtimes_{r} \hA)$ and $M(\hA) \to
  M(C \rtimes_{r} \hA)$, respectively. \qed
\end{corollary}

The reduced crossed product $C \rtimes_{r} \hA$ carries a dual
coaction $\hdelta_{C}$ of $\hAbimod$:
\begin{theorem}
  Let $(K,C,\gamma,\delta_{C})$ be a coaction of $\Abimod$.  Then
  there exists a coaction
  \begin{align} \label{eq:rcp-dual}
    \big(K \rtensorcb H, \, C \rtimes_{r} \hA, \, \gamma \rt
    \hbeta, \, \hdelta_{C}\big)
  \end{align}
  of the concrete Hopf $C^{*}$-bimodule
  $\hAbimod$ such that   for all $c\in C$ and $\ha \in \hA$,
  \begin{align} \label{eq:rcp-dual-coaction}
    \hdelta_{C}\big(\delta_{C}(c)(1 \rtensorh \ha)\big) =
    \big(\delta_{C}(c) \rtensorh 1\big) \big(1 \rtensorh
    \hDelta(\ha)\big).
  \end{align}
  If the coaction $(H,A,\hbeta,\hDelta)$ is  fine, then also the coaction
  \eqref{eq:rcp-dual} is fine.
\end{theorem}
\begin{proof}
 The triple $(K \rtensorcb H, C \rtimes_{r} \hA,
  \gamma \rt \hbeta)$ is a concrete $C^{*}$-$\cbaseosb$-algebra
  because
  \begin{align*}
  \rho_{(\gamma \rt \hbeta)}(\frakB)\big(C \rtimes_{r}
  \hA\big) \subseteq \lnspan (1 \rtensorh
  \rho_{\hbeta}(\frakB)\hA)\delta_{C}(C)\rnspan = C \rtimes_{r} \hA  
  \end{align*}
  by Proposition \ref{proposition:rcp-set}. Moreover, $C \rtimes_{r}
  \hA$ is nondegenerate because $\delta_{C}(C)$ and $1 \rtensorh \hA$
  are nondegenerate (see Remark \ref{remarks:coaction} i)).

  Consider the $*$-homomorphism
  \begin{align*}
    \hdelta_{C} \colon C \rtimes_{r} \hA \to {\cal L}(K \rtensorcb
    \Hsource), \quad
    x \mapsto (1 \rtensorh \checkV)(x \rtensorh 1)(1 \rtensorh
    \checkV^{*}).
  \end{align*}
  Since $\checkV(a\rtensorh 1)\checkV^{*}=a \rtensorh 1$ for all $a
  \in A$ (Lemma \ref{lemma:weak-kac}) and $\delta_{C}(C) \subseteq
  C \hfibrecb M(A)$, 
  \begin{align*}
   (1\rtensorh \checkV)(\delta_{C}(c) \rtensorh 1)(1 \rtensorh
    \checkV^{*}) = \delta_{C}(c) \rtensorh 1 \quad \text{for all } c
    \in C.
  \end{align*}
  Moreover, by Proposition \ref{proposition:balanced-legs},
  $\checkV(\ha \rtensorh 1)\checkV^{*} = \hDelta(\ha)$ for all
  $\ha \in \hA$.  Consequently, Equation \eqref{eq:rcp-dual-coaction} holds
  for all $c \in C$ and $\ha \in \hA$. 

  Let us show that $\hdelta_{C}(C \rtimes_{r} \hA) \subseteq (C
  \rtimes_{r} \hA) \rfibre{\gamma \rt \hbeta}{\alpha} M(\hA)$. By
  Proposition \ref{proposition:rcp-set} and Equation
  \eqref{eq:rcp-dual-coaction},
  \begin{align*}
    \hdelta_{C}(C \rtimes_{r} \hA) = \lnspan \big(\delta_{C}(C)
    \rtensorh 1\big)\big(1 \rtensorh \hDelta(\hA)\big)\rnspan.
  \end{align*}
  Since $\hDelta(\hA)\kalpha{2} \subseteq \lnspan
  \kalpha{2}\hA\rnspan$,
  \begin{gather} \label{eq:rcp-dual-fine}
    \begin{aligned}
      \hdelta_{C}(C \rtimes_{r} \hA) \kalpha{3} &\subseteq \lnspan
      \big(\delta_{C}(C) \rtensorh 1\big)\kalpha{3} (1 \rtensorh \hA)
      \rnspan \\ &= \lnspan \kalpha{3}\delta_{C}(C)(1 \rtensorh
      \hA)\rnspan = \lnspan \kalpha{3}(C \rtimes_{r} \hA)\rnspan.
    \end{aligned}
  \end{gather}
  On the other hand, since $ A\hbeta\subseteq \hbeta$ \cite[Lemma
  4.5]{timmer:cpmu},
  \begin{align*}
    \delta_{C}(C)|\gamma \rt \hbeta\rangle   \subseteq \lnspan (C \hfibrecb
    M(A))|\gamma\rangle_{\leg{1}}\hbeta\rnspan \subseteq \lnspan
    |\gamma\rangle_{\leg{1}} M(A) \hbeta\rnspan =
    |\gamma\rangle_{\leg{1}}\hbeta = |\gamma \rt \hbeta\rangle, 
  \end{align*}
  and therefore,
  \begin{align*}
    \hdelta_{C}(C \rtimes_{r} \hA) |\gamma \rt
    \hbeta\rangle_{\leg{12}} &= \lnspan \big(1 \rtensorh
    \hDelta(\hA)\big)\big(\delta_{C}(C) \rtensorh 1\big) |\gamma \rt
    \hbeta\rangle_{\leg{12}}\rnspan \\ &\subseteq \lnspan \big(1
    \rtensorh \hDelta(\hA)\big) |\gamma \rt
    \hbeta\rangle_{\leg{12}}\rnspan \\ &= \lnspan
    |\gamma\rangle_{\leg{1}} \hDelta(\hA)|\hbeta\rangle_{\leg{1}}\rnspan \subseteq \lnspan
    |\gamma\rangle_{\leg{1}} |\hbeta\rangle_{\leg{1}} \hA\rnspan =
    \lnspan |\gamma \rt \hbeta\rangle_{\leg{12}} \hA\rnspan.
  \end{align*}
  
  Let us show that the $*$-homomorphism $\hdelta_{C}$ is a morphism of
  concrete $C^{*}$-$\cbaseosb$-algebras. By definition of
  $\hdelta_{C}$, 
  \begin{align*}
    \hdelta_{C}(x)  (1 \rtensorh \checkV)|\xi\rangle_{\leg{3}}
    =(1 \rtensorh \checkV)(x \rtensorh 1)|\xi\rangle_{\leg{3}}  =
    (1 \rtensorh \checkV)|\xi\rangle_{\leg{3}}x 
  \end{align*}
  for each $x \in C \rtimes_{r} \hA$ and $\xi \in \hbeta$, and
  therefore,
  \begin{align*}
    \gamma \rt \hbeta \rt \hbeta = \gamma \rt \checkV_{*}(\hbeta \lt
    \hbeta) &=  \lnspan (1 \rtensorh \checkV)\khbeta{3} (\gamma \rt
    \hbeta)\rnspan \\
    &\subseteq \lnspan {\cal L}^{\hdelta_{C}}\big((K
    \hfibrecb H)_{\gamma \rt \hbeta} , (K \hfibrecb \Hsource)_{\gamma
      \rt \hbeta \rt \hbeta}\big) (\gamma \rt \hbeta)\rnspan.
  \end{align*}

  Equation \eqref{eq:rcp-dual-coaction} implies that the $*$-homomorphisms
  \begin{align*}
    (\hdelta_{C} \ast \Id) \circ \hdelta_{C}, (\Id \ast \hDelta) \circ
    \hdelta_{C} \colon C \rtimes_{r} \hA \to {\cal L}\big(K \rtensorcb
    \Hsource \htensor{\hbeta}{\alpha} H\big)
  \end{align*}
  are both given by $\delta_{C}(c)(1 \rtensorh \ha) \mapsto
  \big(\delta_{C}(c) \rtensorh 1 \rtensorh 1\big)\big(1 \rtensorh
  \hDelta^{(2)}(\ha)\big)$ for all $c \in C$, $\ha \in \hA$, where
  $\hDelta^{(2)}:=(\hDelta \ast \Id) \circ \hDelta = (\Id \ast
  \hDelta) \circ \hDelta \colon \hA \to {\cal L}(\Hsource
  \htensor{\hbeta}{\alpha} H)$.

  Summarizing, we find that \eqref{eq:rcp-dual} is a coaction of
  $\hAbimod$ as claimed.  By construction, $\hdelta_{C}$ is injective,
  and Equation \eqref{eq:rcp-dual-fine} shows that this coaction is
  fine if the coaction  $(H,A,\hbeta,\hDelta)$ is  fine.
\end{proof}

\begin{definition}  \label{definition:rcp-dual-coaction}
  Let $(K,C,\gamma,\delta_{C})$ be a coaction of $\Abimod$.   We
  call $\big(K \rtensorcb H, \, C \rtimes_{r} \hA, \, \gamma \rt
  \hbeta, \, \hdelta_{C}\big)$ the associated {\em dual coaction} of
  $\hAbimod$.
\end{definition}
The reduced crossed product construction is functorial in the
following sense:
\begin{proposition}
  Let $\rho$ be a covariant morphism between a coaction
  $(K,C,\gamma,\delta_{C})$ and a coaction $(L,D,\delta,\delta_{D})$ of
  $\Abimod$. Then there exists
  a unique covariant morphism $\rho \rtimes_{r} \Id$ between
  the associated dual coactions such that
  for all $c \in C$ and $\ha \in \hA$,
  \begin{align} \label{eq:rcp-functorial}
    (\rho \rtimes_{r} \Id)\big(\delta_{C}(c)(1 \rtensorh \ha)\big) =
    \delta_{D}(\rho(c))(1 \rtensorh \ha). 
  \end{align}
\end{proposition}
\begin{proof}
  By Remark \ref{remark:rcp-set}, $C \rtimes_{r} \hA \subseteq C
  \hfibrecb {\cal L}(H)$ and $D \rtimes_{r} \hA \subseteq D
  \hfibre{\delta}{\beta} {\cal L}(H)$, and by \cite[Theorem
  3.15]{timmer:cpmu}, there exists a morphism
  \begin{align*}
    \rho \ast \Id \in \Mor\big((C \hfibrecb {\cal L}(H))_{\gamma \rt
      \hbeta}, (D \hfibre{\delta}{\beta} {\cal L}(H))_{\delta \rt
      \hbeta}\big).
  \end{align*}
  Denote by $\rho \rtimes_{r} \Id \colon C \rtimes_{r} \hA \to D
  \hfibre{\delta}{\beta} {\cal L}(H)$ the restriction of $\rho \ast
  \Id$. Then 
  \begin{align*}
    (\rho \rtimes_{r} \Id)\big(\delta_{C}(c)(1 \rtensorh \ha)\big) =
    (\rho \ast \Id)(\delta_{C}(c)) \cdot (\rho \ast \Id)(1 \rtensorh
    \ha) = \delta_{D}(\rho(c))(1 \rtensorh \ha)
  \end{align*}
  for all $c \in C$ and $\ha \in \hA$. In particular, $\rho
  \rtimes_{r} \Id \in \Mor\big((C \rtimes_{r} \hA)_{\gamma \rt
    \hbeta}, (D \rtimes_{r} \hA)_{\delta \rt
    \hbeta}\big)$. Finally, Equations \eqref{eq:rcp-dual-coaction} and
  \eqref{eq:rcp-functorial} imply that $\rho \rtimes_{r} \Id$ is
  covariant with respect to $\hdelta_{C}$ and $\hdelta_{D}$.
\end{proof}
\begin{corollary}
  The assignments $(K,C,\gamma,\delta_{C}) \mapsto \big(K \rtensorcb
  H, \, C \rtimes_{r} \hA, \, \gamma \rt \hbeta, \, \hdelta_{C}\big)$
  and $\rho \mapsto \rho \rtimes_{r} \Id$ define a functor from the
  category of coactions of $\Abimod$ to the category of  coactions
  of $\hAbimod$. \qed
\end{corollary}

\paragraph{Reduced crossed products for coactions of
  $\big(\cbaseosb,H,\hA,\hbeta,\alpha,\hDelta\big)$}

We extend the reduced
crossed product construction to coactions of the concrete Hopf
$C^{*}$-bimodule
$\hAbimod$ as follows:
\begin{definition} \label{definition:rcp-2}
  Let $(K,C,\gamma,\delta_{C})$ be a coaction of $\hAbimod$. The
  associated {\em reduced crossed product} is the $C^{*}$-subalgebra
  $C \rtimes_{\delta_{C},r} A \subseteq {\cal L}(K \rtensorca H)$
  generated by
  \begin{align} \label{eq:rcp-set-2}
    \delta(C)\big(1 \rtensorh \Ad_{U}(A)\big) \subseteq {\cal L}(K
    \rtensorca H).
  \end{align}
  If $\delta_{C}$ is understood, we shortly write $C \rtimes_{r}
  A$ for $C \rtimes_{\delta_{C},r} A$.  
\end{definition}
This construction has the same formal properties like the reduced
crossed product for coactions of $\Abimod$; for the proofs, one simply
replaces the weak $C^{*}$-pseudo-Kac system
$(\alpha,\halpha,\beta,\hbeta,U,V)$ by its predual, applies
Proposition \ref{proposition:balanced-legs}, and uses the results of
the preceding paragraph.
\begin{proposition} \label{proposition:rcp-set-2}
  Let $(K,C,\gamma,\delta_{C})$ be a coaction of
  $\hAbimod$.  Then $C
  \rtimes_{r} A = \lnspan \delta_{C}(C)\big(1 \rtensorh
  \Ad_{U}(A)\big)\rnspan$. \qed
\end{proposition}
\begin{theorem}
  Let $(K,C,\gamma,\delta_{C})$ be a coaction of $\hAbimod$.  There
  exists a coaction $\big(K \rtensorca H, \, C \rtimes_{r} A, \,
  \gamma \rt \halpha, \, \hdelta_{C}\big)$ of $\Abimod$ such that
    \begin{align*}
      \hdelta_{C}\big(\delta_{C}(c)(1 \rtensorh \Ad_{U}(a))\big) =
    \big(\delta_{C}(c) \rtensorh 1\big) \big(1 \rtensorh
    \Ad_{(U \rtensorh 1)}\Delta(a)\big) 
    \end{align*}
    for all $c \in C$ and $a \in A$. If the coaction
    $(H,A,\alpha,\Delta)$ is fine, then also $\big(K \rtensorca H, \, C \rtimes_{r} A,
  \, \gamma \rt \halpha, \, \hdelta_{C}\big)$ is fine.\qed
\end{theorem}
\begin{definition}
  Let $(K,C,\gamma,\delta_{C})$ be a coaction of $\hAbimod$.  Then we
  call $\big(K \rtensorca H, \, C \rtimes_{r} A, \, \gamma \rt
  \halpha, \, \hdelta_{C}\big)$ the associated {\em dual coaction} of
  $\Abimod$.
\end{definition}
\begin{proposition}
 Let $\rho$ be a covariant morphism between a coaction
  $(K,C,\gamma,\delta_{C})$ and a coaction $(L,D,\delta,\delta_{D})$
  of $\hAbimod$. Then
  there exists a unique covariant morphism $\rho \rtimes_{r} \Id$
  between the associated dual coactions such that
  \begin{align*} 
    (\rho \rtimes_{r} \Id)\big(\delta_{C}(c)(1 \rtensorh \Ad_{U}(a))\big) =
    \delta_{D}(\rho(c))(1 \rtensorh \Ad_{U}(a)) 
  \end{align*}
  for all $c \in C$ and $a \in A$.  \qed
\end{proposition}
\begin{corollary}
  The assignments $(K,C,\gamma,\delta_{C}) \mapsto \big(K \rtensorca
  H, \, C \rtimes_{r} A, \, \gamma \rt \halpha, \, \hdelta_{C}\big)$
  and $\rho \mapsto \rho \rtimes_{r} \Id$ define a functor from the
  category of coactions of $\hAbimod$ to the category of coactions of
  $\Abimod$. \qed
\end{corollary}

\section{$C^{*}$-pseudo-Kac systems}

To obtain an analogue of Baaj-Skandalis duality \cite{baaj:2}, we
need to refine the notion of a weak $C^{*}$-pseudo-Kac system
as follows. 

As before, we fix the $C^{*}$-base $\cbasesb$ and the Hilbert space
$H$.
\begin{definition} \label{definition:kac}
  We call a balanced $C^{*}$-pseudo-multiplicative unitary
  $(\alpha,\halpha,\beta,\hbeta,U,V)$ a {\em
    $C^{*}$-pseudo-Kac-system} if
  \begin{enumerate}
  \item  the $C^{*}$-pseudo-multiplicative unitaries $V,\checkV,\hatV$
    are regular, and
  \item $(\Sigma(1 \rtensorh U)V)^{3} = 1$ in ${\cal L}(\Hsource)$.
  \end{enumerate}
\end{definition}
\begin{remark} \label{remark:kac}
 In leg notation, the equation $\big(\Sigma(1 \rtensorh
     U)V\big)^{3}=1$ takes the form $(\Sigma U_{\leg{2}} V)^{3} = 1$.
     Conjugating by $\Sigma$ or $V$, we find that this condition is
     equivalent to the relation $( U_{\leg{2}}V\Sigma)^{3}=1$ and to the
     relation $(V\Sigma U_{\leg{2}})^{3}=1$.
\end{remark}
\begin{remark}
  Definition \ref{definition:kac} greatly simplifies the definition of
  a pseudo-Kac system on $C^{*}$-modules given in
  \cite{timmermann} and still covers our main examples. The two
  concepts are related as follows.

  Let $(\alpha,\halpha,\beta,\hbeta,U,V)$ be a $C^{*}$-pseudo-Kac
  system and assume that $\frakBo \cong \frakB^{op}$. Then the
  $C^{*}$-modules $\alpha,\halpha,\beta,\hbeta$, the representations
  $\rho_{\alpha}$, $\rho_{\halpha}$, $\rho_{\beta}$, $\rho_{\hbeta}$,
  the unitaries $U_{\alpha} \colon \alpha \to \halpha$, $U_{\halpha}
  \colon \halpha \to \alpha$, $U_{\beta} \colon \beta \to \hbeta$, $
  U_{\hbeta}\colon \hbeta \to \beta$ together with the family of
  unitaries
    \begin{align*}
      V_{\alpha \lt \alpha} &\colon \alpha {_{\rho_{\hbeta}}\tl}
      \alpha \to \alpha {\tr_{\rho_{\beta}}} \alpha, & V_{\hbeta \rt \beta}
      &\colon \hbeta \tr_{\rho_{\alpha}} \beta \to \hbeta
      {_{\rho_{\alpha}}\tl} \beta, & V_{\halpha \lt \alpha} &\colon
      \halpha {_{\rho_{\hbeta}}\tl} \alpha \to \halpha
      {_{\rho_{\alpha}}\tl} \beta, \\ V_{\hbeta \rt \hbeta} &\colon
      \hbeta \tr_{\rho_{\alpha}} \hbeta \to \alpha \tr_{\rho_{\beta}}
      \hbeta, & V_{\beta \lt \alpha} &\colon \beta {_{\rho_{\hbeta}}\tl}
      \alpha \to \beta {_{\rho_{\alpha}}\tl} \beta, & V_{\hbeta \rt
        \halpha} &\colon \hbeta \tr_{\rho_{\alpha}} \halpha \to \alpha
      \tr_{\rho_{\beta}} \halpha
    \end{align*}
    form a pseudo-Kac system in the sense of \cite[Definition
    2.53]{timmermann}. 
\end{remark}
We shall use the following reformulation of condition ii) in 
Definition \ref{definition:kac}:
\begin{lemma} \label{lemma:kac-condition}
    Let $(\alpha,\halpha,\beta,\hbeta,U,V)$ be a balanced
  $C^{*}$-pseudo-multiplicative unitary. Then
  $(\Sigma U_{\leg{2}}V)^{3}=1$ if and only if $\hatV V\checkV =
  U_{\leg{1}}\Sigma$.
\end{lemma}
\begin{proof}
  Rearranging the factors in the product $U_{\leg{1}}U_{\leg{2}}\big(\Sigma
    U_{\leg{2}}V\big)^{3} U_{\leg{2}}\Sigma$, we find
   \begin{align*}
     U_{\leg{1}} U_{\leg{2}} (\Sigma U_{\leg{2}}V)^{3} U_{\leg{2}}
     \Sigma = \Sigma U_{\leg{1}} V U_{\leg{1}} \Sigma \cdot V \cdot
     \Sigma U_{\leg{2}} V U_{\leg{2}} \Sigma = \hatV \cdot V \cdot
     \checkV.
   \end{align*}
   Thus, $(\Sigma U_{\leg{2}} V)^{3}$ is equal to $1$ if and only if
   $U_{\leg{1}}\Sigma=U_{\leg{1}}U_{\leg{2}}U_{\leg{2}}\Sigma$ is
   equal to $\hatV V \checkV$.
\end{proof}
The notion of a $C^{*}$-pseudo-Kac system is symmetric in the
following sense:
\begin{proposition}
  Let $(\alpha,\halpha,\beta,\hbeta,U,V)$ be a $C^{*}$-pseudo-Kac
  system. Then also the tuples \eqref{eq:weak-kac-symmetry} are
  $C^{*}$-pseudo-Kac systems.
\end{proposition}
\begin{proof}
  Equations \eqref{eq:balanced-iterate}, \eqref{eq:balanced-op} and
 \cite[Remark 4.11]{timmer:cpmu} imply that the tuples
  \eqref{eq:weak-kac-symmetry} satisfy condition i) of Definition
  \ref{definition:kac}.  To check that they also satisfy condition
  ii), we use Remark \ref{remark:kac}:
  \begin{gather*}
    (\Sigma U_{\leg{2}} \hatV)^{3} = (V \Sigma U_{\leg{2}})^{3} = 1,  \qquad
    (\checkV\Sigma U_{\leg{2}} )^{3} =(\Sigma U_{\leg{2}}V)^{3} = 1,
    \\
    (U_{2}V^{op}\Sigma)^{3} = (U_{\leg{2}}\Sigma V^{*})^{3} = 
    \big((V\Sigma U_{\leg{2}})^{3}\big)^{*}=
    1. \qedhere
  \end{gather*}
\end{proof}

The following two propositions are essential for our  duality theorem:
\begin{proposition} \label{proposition:kac-weak}
    Let $(\alpha,\halpha,\beta,\hbeta,U,V)$ be a $C^{*}$-pseudo-Kac
    system. Then it is a weak $C^{*}$-pseudo-Kac system.
\end{proposition}
\begin{proof}
  Since $V$ is regular, it is well-behaved. We show that condition
  i)(a) in Lemma \ref{lemma:weak-kac} holds. By Lemma
  \ref{lemma:balanced}, $\checkV_{\leg{12}}V_{\leg{13}} =
  V_{\leg{13}}V_{\leg{23}}\checkV_{\leg{12}}$. We multiply this
  equation by $V_{\leg{12}}$ on the left and by $\Sigma_{\leg{12}}$ on
  the right, use the pentagon equation for $V$, and obtain
    \begin{align*}
      V_{\leg{12}}\checkV_{\leg{12}}\Sigma_{\leg{12}} V_{\leg{23}} =
      V_{\leg{12}}\checkV_{\leg{12}}V_{\leg{13}}\Sigma_{\leg{12}}  =
      V_{\leg{12}}V_{\leg{13}}V_{\leg{23}}\checkV_{\leg{12}}\Sigma_{\leg{12}}  =
      V_{\leg{23}}V_{\leg{12}}\checkV_{\leg{12}}\Sigma_{\leg{12}}.
    \end{align*}
    By Lemma \ref{lemma:kac-condition}, we can replace
    $V_{\leg{12}}\checkV_{\leg{12}}\Sigma_{\leg{12}}$ by
    $\hatV_{\leg{12}}^{*}U_{\leg{1}}$, and find that
    $\hatV_{\leg{12}}^{*}U_{\leg{1}}V_{\leg{23}} =
    V_{\leg{23}}\hatV_{\leg{12}}^{*}U_{\leg{1}}$.  Since
    $\hatV_{\leg{12}}$ is unitary and
    $U_{\leg{1}}V_{\leg{23}}=V_{\leg{23}}U_{\leg{1}}$, we can conclude
    $\hatV_{\leg{12}}V_{\leg{23}}=V_{\leg{23}}\hatV_{\leg{12}}$, that
    is, condition i)(a) in Lemma \ref{lemma:weak-kac} holds.
    A similar argument shows that condition ii)(a) in Lemma
    \ref{lemma:weak-kac} holds.
\end{proof}

\begin{proposition} \label{proposition:kac-compact}
  Let $(\alpha,\halpha,\beta,\hbeta,U,V)$ be a $C^{*}$-pseudo-Kac
  system. Then $\lnspan A(V)\hA(V)\rnspan = \lnspan \halpha\halpha^{*}
  \rnspan$.
\end{proposition}
We need the  following lemma:
\begin{lemma} \label{lemma:kac-aux} Let $V \colon \Hsource \to
  \Hrange$ be a regular $C^{*}$-pseudo-multiplicative unitary. Then
  $\lnspan V\kalpha{2}\hA(V) \rnspan = \lnspan
  \kbeta{2}\hA(V)\rnspan$.
\end{lemma}
\begin{proof}
 The first commutative diagram in \cite[Proof of
  Proposition 4.12]{timmer:cpmu} shows 
  \begin{align*}
    \lnspan V\kalpha{2}\hA(V) \rnspan =
    \lnspan V\kalpha{2}\hA(V)^{*} \rnspan
    = \lnspan \balpha{2}V_{\leg{12}}^{*} \kbeta{3}\kbeta{2}\rnspan,
  \end{align*}
  and the following commutative diagram shows that this  equals
  $\lnspan \kbeta{2} \hA(V)^{*}\rnspan = \lnspan \kbeta{2}
  \hA(V)\rnspan$:
  \begin{gather*}\smalldiagram
    \xymatrix@C=35pt{
      {H} \ar[r]^(0.42){\kbeta{2}}  \ar `d/0pt[dd]^{\hA(V)^{*}} [ddr] &
      {\Hrange}  \ar[r]^(0.42){\kbeta{3}} \ar[d]^{V^{*}} &
      {\Hthree} \ar[d]^{V_{\leg{12}}^{*}} \\
       & {\Hsource} \ar[r]^(0.42){\kbeta{3}}
      \ar[d]^{\balpha{2}} &
      {\Hfour}  \ar[d]^{\balpha{2}} \\
      & H \ar[r]^(0.42){\kbeta{2}}&  {\Hrange}.
    } \qedhere
  \end{gather*}
\end{proof}

\begin{proof}[Proof of Proposition \ref{proposition:kac-compact}]
  Since $V$ is regular and $V^{*}=\Sigma U_{\leg{2}} V\Sigma
  U_{\leg{2}} V\Sigma U_{\leg{2}}$, 
  \begin{align*}
    \lnspan \halpha \halpha^{*}\rnspan = \lnspan
    U\alpha\alpha^{*}U\rnspan  & =
    \lnspan U \balpha{2} V^{*}\kalpha{1} U\rnspan \\ &=
    \lnspan U \balpha{2}\Sigma U_{\leg{2}} V\Sigma
    U_{\leg{2}} V\Sigma U_{\leg{2}} \kalpha{1} U \rnspan 
    = \lnspan \balpha{1} V\Sigma U_{\leg{2}}V\kalpha{2}\rnspan. 
  \end{align*}
  We multiply  on the right by $\hA:=\hA(V)$, use Equation
 \eqref{eq:ha-halpha} and Lemma \ref{lemma:kac-aux}, and find
 \begin{align*}
    \lnspan \halpha \halpha^{*}\rnspan =  \lnspan \halpha \halpha^{*}
    \hA\rnspan &=
 \lnspan \balpha{1} V\Sigma U_{\leg{2}}V\kalpha{2} \hA\rnspan \\ &=
 \lnspan \balpha{1} V\Sigma U_{\leg{2}}\kbeta{2} \hA\rnspan  =
 \lnspan \balpha{1} V\khbeta{1} \hA\rnspan = \lnspan A(V)
 \hA(V)\rnspan. \qedhere
 \end{align*}
\end{proof}

\paragraph{The duality theorem}
Let $(\alpha,\halpha,\beta,\hbeta,U,V)$ be a $C^{*}$-pseudo-Kac
system, put
\begin{align*}
    \hA&:=\hA(V), &\hDelta&:=\hDelta_{V}, & A&:=A(V),
    &\Delta:=\Delta_{V},
\end{align*}
and write $\hAbimod$ and
$\Abimod$ for the concrete Hopf
$C^{*}$-bimodules $\big(\cbaseosb,H,\hA,\hbeta,\alpha,\hDelta\big)$ and
$\big(\cbasesb,H, A,\alpha,\beta,\Delta\big)$, respectively.

Before we can state the duality theorem, we need some preliminaries.
Let $(K,C,\gamma,\delta_{C})$ be a coaction of $\Abimod$.  By
Lemma \ref{lemma:ind-functorial}, $\lnspan \kbeta{2}C\bbeta{2}\rnspan
\subseteq {\cal L}(K \rtensorcb H)$  is a $C^{*}$-algebra and there
exists a $*$-homomorphism
\begin{align*}
  \Ind_{\beta}(\delta_{C}) \colon \lnspan \kbeta{2}C\bbeta{2}\rnspan
  \to \lnspan \kbeta{3}\delta_{C}(C)\bbeta{3}\rnspan \subseteq {\cal
    L}(K \rtensorcb \Hrange).
\end{align*} 
Likewise, if $(K,C,\gamma,\delta_{C})$ is a coaction of $\hAbimod$,
then
$\lnspan \kalpha{2}C\balpha{2}\rnspan \subseteq {\cal L}(K \rtensorca
H)$ is a $C^{*}$-algebra and there exists a $*$-homomorphism
\begin{align*}
  \Ind_{\alpha}(\delta_{C}) \colon \lnspan
  \kalpha{2}C\balpha{2}\rnspan \to \lnspan
  \kalpha{3}\delta_{C}(C)\balpha{3} \rnspan \subseteq {\cal L}(K
  \rtensorca \Hsource).
\end{align*}
\begin{theorem}
  \begin{enumerate}
  \item Let $(K,C,\gamma,\delta_{C})$ be a fine coaction of
    $\Abimod$.  There exists a natural isomorphism
    \begin{align*}
      C \rtimes_{r} \hA \rtimes_{r} A \cong \lnspan
      \kbeta{2} C  \bbeta{2}\rnspan \subseteq {\cal L}(K
      \rtensorcb H)
    \end{align*}
    which identifies the bidual coaction
    $\widehat{\widehat{\delta}}_{C}$ on $C \rtimes_{r} \hA \rtimes_{r}
    A$ with the map
    $\Ad_{(1 \rtensorh \Sigma\hatV)} \circ \Ind_{\beta}(\delta_{C})$.
\item Let $(K,C,\gamma,\delta_{C})$ be a fine coaction of
  $\hAbimod$.  There exists a natural isomorphism
    \begin{align*}
      C \rtimes_{r} A \rtimes_{r} \hA \cong \lnspan \kalpha{2} C
      \balpha{2}\rnspan \subseteq {\cal L}(K \rtensorca H)
    \end{align*}
    which identifies the bidual coaction
    $\widehat{\widehat{\delta}}_{C}$ on $C \rtimes_{r} A \rtimes_{r}
    \hA$ with the map $\Ad_{(1 \rtensorh \Sigma V)} \circ
    \Ind_{\alpha}(\delta_{C})$.
  \end{enumerate}
\end{theorem}
\begin{proof}
  i) By Proposition \ref{proposition:rcp-set-2}, Definition
  \ref{definition:rcp-dual-coaction}, and Proposition
  \ref{proposition:rcp-set}, the iterated crossed product $C
  \rtimes_{r} \widehat{A} \rtimes_{r} A$ is equal to
   \begin{align*}
     \lnspan \big(\delta_{C}(C) \rtensorh 1\big)
     \big(1 \rtensorh \hDelta(\hA)\big)
     \big(1 \rtensorh 1 \otimes \Ad_{U}(A)\big) \rnspan \subseteq {\cal
     L}(K \rtensorcb \Hsource).
   \end{align*}
   By definition of $\Delta$ and $\widehat{\Delta}$ and by
   Lemma \ref{lemma:weak-kac}, conjugation by 
   $1 \rtensorh V$ maps this $C^{*}$-algebra isomorphically onto
   \begin{align*}
     \lnspan \delta_{C}^{(2)}(C) \big(1 \rtensorh 1 \rtensorh
     \widehat{A}\big)\big(1 \rtensorh 1 \rtensorh \Ad_{U}(A)\big)
     \rnspan \subseteq {\cal L}(K \rtensorcb \Hrange),
   \end{align*}
   where $\delta_{C}^{(2)} = (\Id\ast \Delta) \circ \delta_{C} = (\delta_{C}
   \ast \Id) \circ \delta_{C}$. Since $\delta_{C}$ is injective, the formula
   \begin{align*}
     \delta_{C}^{(2)}(c)(1 \rtensorh 1 \rtensorh T) \mapsto
     \delta_{C}(c)(1 \rtensorh T), \quad\text{where } c \in C, \ T \in
     {\cal L}(H),
   \end{align*}
   defines an isomorphism of this $C^{*}$-algebra onto
   \begin{align} \label{eq:dual-1}
     \lnspan \delta_{C}(C) \big(1 \rtensorh \widehat{A} \cdot
     \Ad_{U}(A)\big)\rnspan \subseteq {\cal L}(K \rtensorcb H).
   \end{align}
   By Proposition \ref{proposition:balanced-legs} and Proposition
   \ref{proposition:kac-compact}, applied to the $C^{*}$-pseudo-Kac
   system $(\hbeta,\beta,\alpha,\halpha,U,\checkV)$,  we have $\lnspan
   \widehat{A} \Ad_{U}(A)\rnspan = \lnspan A(\checkV)
   \widehat{A}(\checkV)\rnspan = \lnspan \beta \beta^{*}\rnspan$.
   We insert this relation into \eqref{eq:dual-1}, use the fact that
   $\delta_{C}$ is a fine coaction, and find
   \begin{align*}
     C \rtimes_{r} \widehat{A} \rtimes_{r} A \cong
     \lnspan \delta(C) \kbeta{2}\bbeta{2}\rnspan = \lnspan \kbeta{2} C
     \bbeta{2}\rnspan \subseteq {\cal L}(K \rtensorcb H).
   \end{align*}
   This isomorphism identifies an element of the form
   \begin{gather}\label{eq:kac-one}
     \big(\delta_{C}(c) \rtensorh 1\big)\big(1 \rtensorh
     \hDelta(\ha)\big)\big(1 \rtensorh 1 \rtensorh
     \Ad_{U}(a)\big)  \in  C\rtimes_{r} \widehat{A}
     \rtimes_{r} A
     \shortintertext{with}
     \label{eq:kac-two}
     \delta_{C}(c)\big(1 \rtensorh \hat{a}\cdot \Ad_{U}(a)\big)
       \in  \lnspan \kbeta{2}C\bbeta{2}\rnspan.
   \end{gather}
   The bidual coaction $\widehat{\widehat{\delta}}_{C}$  maps the
   element \eqref{eq:kac-one} to
   \begin{align*}
     \big(\delta_{C}(c) \rtensorh 1 \rtensorh 1\big)\big(1 \rtensorh
     \widehat{\Delta}(\hat{a}) \rtensorh 1\big)\big(1 \rtensorh 1
     \rtensorh \Ad_{(U \rtensorh 1)}(\Delta(a))\big),
   \end{align*}
   and the map $\Ad_{(1 \rtensorh \Sigma \hatV)} \circ
   \Ind_{\beta}(\delta_{C})$ sends the element \eqref{eq:kac-two} to
   \begin{multline*}
     \Ad_{(1 \rtensorh \Sigma\hatV)}\big(\delta_{C}^{(2)}(c) (1
     \rtensorh 1 \rtensorh \hat{a} \cdot \Ad_{U}(a))\big) =
     \big(\delta_{C}(c) \rtensorh 1\big)\big(1 \rtensorh \hat{a} \rtensorh
     1\big)\big(1 \rtensorh \Ad_{(U\rtensorh 1)}(\Delta(a))\big);
   \end{multline*}
   here, we used the following relations:
   \begin{itemize}
   \item $\Ad_{\Sigma \hatV}(\Delta(b)) = \Ad_{\Sigma}(1 \rtensorh
     b)=b \rtensorh 1$ for all $b \in A$ by Proposition
     \ref{proposition:balanced-legs};
   \item $\Ad_{\Sigma \hatV}(1 \rtensorh \ha)= \Ad_{\Sigma}(1
     \rtensorh \ha)= \ha \rtensorh 1$ by Lemma
     \ref{lemma:weak-kac};
   \item $\Ad_{\Sigma\hatV}(1 \rtensorh \Ad_{U}(a)) = \Ad_{(U
       \rtensorh 1)}\!\big(\Ad_{V}(1 \rtensorh a)\big) = \Ad_{(U
       \rtensorh 1)}(\Delta(a))$ by definition of $\Delta$.
   \end{itemize} 
   Therefore, the bidual coaction $\widehat{\widehat{\delta}}_{C}$ on
   $C \rtimes_{r} \widehat{A}\rtimes_{r} A$ corresponds to the map
   $\Ad_{(1 \rtensorh \Sigma\hatV)} \circ \Ind_{\beta}(\delta_{C})$.

   \smallskip
   
   ii) The proof is similar to the proof of i), simply replace the
   $C^{*}$-pseudo-Kac system
$(\alpha,\halpha,\beta,\hbeta,U,V)$ by its predual 
$(\hbeta,\beta,\alpha,\halpha,U,\checkV)$.
 \end{proof}

\section{The $C^{*}$-pseudo-Kac system of a locally compact
  groupoid}

The prototypical example of a $C^{*}$-pseudo-Kac system is the one
associated to a locally compact groupoid. The underlying
$C^{*}$-pseudo-multiplicative unitary was described in
\cite{timmer:cpmu}.  For background on groupoids, Haar measures, and
quasi-invariant measures, see \cite{renault} or \cite{paterson}.

\paragraph{The data}
Let $G$ be a locally compact, Hausdorff, second countable groupoid. We
denote its unit space by $G^{0}$, its range map by $r_{G}$, its source
map by $s_{G}$, and put $G^{u}:=r_{G}^{-1}(\{u\})$,
$G_{u}:=s_{G}^{-1}(u)$ for each $u \in G^{0}$. 

We assume that $G$ has a left Haar system $\lambda$, and denote the
associated right Haar system by $\lambda^{-1}$. Let $\mu$ be a measure
on $G^{0}$ and denote by $\nu$ the measure on $G$ given by
 \begin{align*}
   \int_{G} f \,d \nu &:= \int_{G^{0}} \int_{G^{u}} f(x) \, d\lambda^{u}(x)
\,   d\mu(u) \quad \text{for all } f \in C_{c}(G).
 \end{align*}
 The push-forward of $\nu$ via the inversion map $G \to G$, $x \mapsto
 x^{-1}$, is denoted by $\nu^{-1}$; evidently,
 \begin{align*}
   \int_{G} f d\nu^{-1} &= \int_{G^{0}}  \int_{G_{u}} f(x)
   d\lambda^{-1}_{u}(x) \, d\mu(u).
 \end{align*}
 We assume that the measure $\mu$ is quasi-invariant, i.e., that $\nu$
 and $\nu^{-1}$ are equivalent. Note that there always exist
 sufficiently many quasi-invariant measures \cite{renault}.  We denote
 by $D:=d\nu/d\nu^{-1}$ the Radon-Nikodym derivative. 

\paragraph{The $C^{*}$-base and the Hilbert space}
 The measure $\mu$ defines a tracial proper weight on the
 $C^{*}$-algebra $C_{0}(G^{0})$, which we denote by $\mu$
 again. Moreover, we denote by $\cbasesb$ the $C^{*}$-base associated
 to $\mu$ (see Section \ref{section:basics}); thus,
 $\frakH=L^{2}(G^{0},\mu)$. Note that $\cbasesb=\cbaseosb$ because
 $C_{0}(G^{0})$ is commutative.
 Put $H:=L^{2}(G,\nu)$.

\paragraph{The $C^{*}$-factorizations}
 The space $C_{c}(G)$ forms a pre-$C^{*}$-module over $C_{0}(G^{0})$
 with respect to the structure maps
 \begin{align*}
   \begin{aligned}
     \langle \xi'|\xi\rangle(u)&= \int_{G^{u}}
     \overline{\xi'(x)}\xi(x) d\lambda^{u}(x), & (\xi f)(x) &=
     \xi(x)f(r_{G}(x)),
   \end{aligned}
 \end{align*}
and also with respect to the structure maps
\begin{align*}
  \begin{aligned}
    \langle \xi'|\xi\rangle(u)&= \int_{G_{u}} \overline{\xi'(x)}\xi(x)
    d\lambda^{-1}_{u}(x), & (\xi f)(x) &= \xi(x)f(s_{G}(x)).
  \end{aligned}
\end{align*}
 Denote the completions of these  pre-$C^{*}$-modules by
$L^{2}(G,\lambda)$ and $L^{2}(G,\lambda^{-1})$, respectively.
By \cite[Proposition 5.1]{timmer:cpmu},
  there exist isometric embeddings
  \begin{align*}
    j\colon L^{2}(G,\lambda) &\to {\cal L}(\frakH,H) &&\text{and}& \hat j \colon
    L^{2}(G,\lambda^{-1}) &\to {\cal L}\big(\frakH,H\big)
  \end{align*}
  such that for all $\xi \in C_{c}(G)$,  $\zeta \in L^{2}(G^{0},\mu)$, $x \in G$,
  \begin{align*}
    \big(j(\xi) \zeta\big)(x) &= \xi(x)\zeta(r_{G}(x)), &
    \big(\hat j(\xi)\zeta\big)(x) &= \xi(x) D^{-1/2}(x) \zeta(s_{G}(x)).
  \end{align*}
  Moreover, by \cite[Proposition 5.1]{timmer:cpmu}, the images
  \begin{align*}
    \alpha:=\beta:=j(L^{2}(G,\lambda)) \quad \text{and} \quad
    \halpha:=\hbeta := \hat j(L^{2}(G,\lambda^{-1}))
  \end{align*}
  are compatible $C^{*}$-factorizations of $H$ with respect to
  $\cbasesb$, the maps $j$ and $\hat j$ are unitary as maps of
  $C^{*}$-modules over $C_{0}(G^{0}) \cong \frakB$, and for all $x \in
  G$, $\xi \in C_{c}(G)$, and $f \in C_{0}(G^{0})$,
  \begin{align*}
    \big(\rho_{\alpha}(f)\xi\big)(x) &:=
    f\big(r_{G}(x)\big)\xi(x), & 
    \big(\rho_{\halpha}(f)\xi\big)(x) &:=
    f\big(s_{G}(x)\big) \xi(x)
  \end{align*}

\paragraph{The $C^{*}$-pseudo-multiplicative unitary}

  By \cite{vallin:2} and \cite[Proposition 2.14]{timmer:cpmu}, the
  Hilbert spaces $\Hsource$ and $\Hrange$ can be described as follows.
  Define a measure $\nu^{2}_{s,r}$ on $G^{2}_{s,r}:=\{(x,y)\in G
  \times G \mid s(x)=r(y)\}$ by
 \begin{gather*}
   \int_{G^{2}_{s,r}} \!\! f\, d\nu^{2}_{s,r} := \int_{G^{0}}
   \int_{G^{u}} \int_{G^{s_{G}(x)}} f(x,y) \, d\lambda^{s_{G}(x)}(y)
   \, d\lambda^{u}(x) \, d\mu(u),
 \end{gather*}
 and a measure $\nu^{2}_{r,r}$ on $G^{2}_{r,r}:=\{ (x,y) \in
 G^{2}\mid r_{G}(x)=r_{G}(y)\}$  by
 \begin{gather*}
   \int_{G^{2}_{r,r}} \!\! g\, d\nu^{2}_{r,r} := \int_{G^{0}}
   \int_{G^{u}}\int_{G^{u}} g(x,y)\, d\lambda^{u}(y)\, d\lambda^{u}(x)\,
   d\mu(u),
 \end{gather*}
where $f \in C_{c}(G^{2}_{s,r})$ and $g\in C_{c}(G^{2}_{r,r})$. By
\cite{vallin:2} and \cite[Theorem 5.2]{timmer:cpmu},
there exist isomorphisms
\begin{align*}
  \Hsource &\cong
    L^{2}\big(G^{2}_{s,r},\nu^{2}_{s,r}\big), &
    \Hrange &\cong    
    L^{2}\big(G^{2}_{r,r},\nu^{2}_{r,r}\big)
\end{align*}
and a $C^{*}$-pseudo-multiplicative unitary $V \colon \Hsource \to
\Hrange$ such that, with respect to these isomorphisms, $(V\zeta)(x,y)
= \zeta(x,x^{-1}y)$ for all $\zeta \in L^{2}(G^{2}_{s,r},\,
\nu^{2}_{s,r})$ and $(x,y) \in G^{2}_{r,r}$.

\paragraph{The $C^{*}$-pseudo-Kac system}

By definition of the Radon-Nikodym derivative $D=d\nu/d\nu^{-1}$,
there exists a unitary $U \in {\cal L}(H)$  such that
\begin{align*}
  (U \xi)(x) := \xi(x^{-1}) D(x)^{-1/2} \quad \text{for all } x \in
 G, \xi \in C_{c}(G).
\end{align*}
This unitary is a symmetry because for all $\xi \in C_{c}(G)$ and
$\nu$-almost all $x \in G$,
\begin{align*}
  (U^{2}\xi) (x) = (U\xi)(x^{-1}) D(x)^{-1/2} = \xi(x) D(x)^{1/2}
  D(x)^{-1/2} = \xi(x);
\end{align*}
here, we used the relation $D(x)^{-1}=D(x^{-1})$
\cite[p.\ 23]{renault}, \cite[Eq.\ (3.7)]{paterson}.

\begin{theorem}
  $(\alpha,\halpha,\beta,\hbeta,V,U)$ is a $C^{*}$-pseudo-Kac system.
\end{theorem}
\begin{proof}
  By \cite[Proposition 5.3]{timmer:cpmu}, the
  $C^{*}$-pseudo-multiplicative unitary $V$ is regular.

  We claim that $\hatV=\Sigma U_{\leg{1}}VU_{\leg{1}}\Sigma$ is equal to
  $V^{op}=\Sigma V^{*} \Sigma$.  Indeed, for all $\zeta \in
  L^{2}(G^{2}_{r,r},\nu^{2}_{r,r})$ and  $(x,y) \in G^{2}_{s,r}$,
  \begin{align*}
    (U_{\leg{1}}VU_{\leg{1}}  \zeta)(x,y) 
    &= (VU_{\leg{1}}  \zeta)(x^{-1},y)D(x)^{-1/2} \\
    &= (U_{\leg{1}}  \zeta)(x^{-1},xy)D(x)^{-1/2} 
    = \zeta(x,xy) D(x^{-1})^{-1/2} D(x)^{-1/2}.
  \end{align*}
  Since $ D(x^{-1})^{-1/2} D(x)^{-1/2}=1$ for $\nu$-almost all $x
  \in G$ \cite[p.\ 23]{renault}, \cite[Eq.\ (3.7)]{paterson}, we can
  conclude $U_{\leg{1}}VU_{\leg{1}}=V^{*}$, and the claim follows.

  In particular, $\hatV=V^{op}$ is a regular
  $C^{*}$-pseudo-multiplicative unitary \cite[Remark
  4.11]{timmer:cpmu}.

  To finish the proof, we only need to show that the map $Z:=\Sigma
  U_{\leg{2}}V \colon \Hsource \to \Hsource$ satisfies
  $Z^{3}=1$. For all $\zeta \in L^{2}(G^{2}_{s,r},\nu^{2}_{s,r})$
  and $(x,y) \in G^{2}_{s,r}$,
  \begin{align*}
    (\Sigma U_{\leg{2}} V \zeta)(x,y) = (V\zeta)(y,x^{-1})D(x)^{-1/2} =
    \zeta(y,y^{-1}x^{-1}) D(x)^{-1/2},
  \end{align*}
  and therefore,
  \begin{align*}
    (Z^{3}\zeta)(x,y) &= (Z^{2}\zeta) (y,y^{-1}x^{-1})
    D(x)^{-1/2} \\
    &= (Z\zeta)(y^{-1}x^{-1},x y y^{-1}) \big(D(x)D(y)\big)^{-1/2}
    \\
    &= \zeta(x,x^{-1} xy) \big(D(x)D(y)D(y^{-1}x^{-1})\big)^{-1/2}.
  \end{align*}
  Since $D(x)D(y)D(y^{-1}x^{-1}) = 1$ for $\nu^{2}_{s,r}$-almost all
  $(x,y)\in G^{2}_{s,r}$ \cite{hahn} (but see also
  \cite[p. 89]{paterson}), we can conclude $Z=\Id$.
  \end{proof}

\def\cprime{$'$}

\end{document}